\documentclass[11pt,reqno]{amsart}

\usepackage{amsmath,amssymb,amsthm,mathtools, comment}
\usepackage[hidelinks]{hyperref}
\usepackage[margin=1in]{geometry}

\newtheorem{theorem}{Theorem}
\newtheorem{proposition}[theorem]{Proposition}
\newtheorem{lemma}[theorem]{Lemma}
\newtheorem{corollary}[theorem]{Corollary}
\newtheorem{remark}[theorem]{Remark}

\newcommand{\Z}{\mathbb Z}
\newcommand{\C}{\mathbb C}
\newcommand{\eps}{\varepsilon}
\newcommand{\ord}{\operatorname{ord}}

\title{Improved average and almost-all bounds for $G(n)$}
\author{Chiara Bellotti}
\address{Max Planck Institute for Mathematics, Bonn, Germany}
\email{bellotti@mpim-bonn.mpg.de}
\date{}

\usepackage{xcolor}
\definecolor{carmine}{rgb}{0.7, 0.1, 0.09}

\keywords{generating sets modulo $n$, character sums,
large sieve, average order}
\subjclass[2020]{Primary 11N37; Secondary 11L40, 11N35, 11A07.}
\begin{document}
\begin{abstract}
In this paper we study the least positive integer $G(n)$ such that the integers $a\leq G(n)$ with $(a,n)=1$ generate $(\mathbb Z/n\mathbb Z)^\times$.  We prove that, for every $\varepsilon>0$,
\[
\sum_{n\leq x}G(n)\ll_\varepsilon x(\log x)^{8/3+\varepsilon},
\]
improving the previously known bound $\ll x(\log x)^{97}$.  We also prove that $G(n)\leq(\log n)^2$ for almost all $n$, unconditionally. Thus, for almost all $n$, we obtain the same logarithmic exponent $2$ as in the classical pointwise bound under the Generalized Riemann Hypothesis (GRH). Lastly, we also obtain quantitative exceptional-set estimates with logarithmic savings for general moduli  and power savings for prime moduli, including an application to the least  prime primitive root. Under GRH,  we further strengthen the exceptional-set estimate for general moduli and obtain a mean-value upper bound with an asymptotic coefficient of $4/e^2$.

\end{abstract}

\maketitle

\section{Introduction}

Given a positive integer $n$, let $ \Gamma_n:=(\Z/n\Z)^\times$, and, when $\Gamma_n$ is nontrivial, let $G(n)$ denote the least positive integer $G$ such that the residue classes of the integers $a\leq G$ with $(a,n)=1$ generate $\Gamma_n$. The problem of bounding $G(n)$ is closely related to questions in deterministic primality testing and has been studied both pointwise and on average. Related questions concerning generation by small primes have been studied by Klurman, Shparlinski and Ter\"av\"ainen~\cite{KSTGeneration}, who established an almost-all result showing that $(\Z/n\Z)^\times$ can be generated by subset products of suitably chosen small primes.

Assuming the Generalized Riemann Hypothesis, Montgomery \cite{monties} proved that $G(n)\ll(\log n)^2$, and Bach \cite{ebach} later obtained an effective version of this estimate.  Bach~\cite{ebach} established under the same assumption that $G(n)<2(\log n)^2$ for $n\ge 2$. Lamzouri, Li and Soundararajan~\cite{LLS15} proved, also under GRH, that if $n\ge 3000$ and $n$ has no prime divisors below $(\log n)^2$, then we even have $G(n)\le (\log n)^2$. Unconditionally, Bach and Huelsbergen \cite{BH} obtained a bound of order essentially $n^{1/2}$, which was subsequently improved by Burthe \cite{Bu1} and then by Norton \cite{Norton1998}, who proved that $ G(n)\ll_\varepsilon n^{1/(4\sqrt e)+\varepsilon}$. These pointwise estimates are closely connected with bounds for the least integer on which a nonprincipal Dirichlet character takes a value different from $0$ and $1$.

The average value is expected to be considerably smaller. Bach and Huelsbergen \cite{BH} conjectured, in particular, that
\[
\frac1x\sum_{n\leq x}G(n)\sim(\log\log x)\log\log\log x.
\]
The strongest unconditional result previously available in this direction is due to Burthe \cite{Burthe1997}, who used zero-density estimates for Dirichlet $L$-functions to prove $\sum_{n\leq x}G(n) \ll x(\log x)^{97}$.
    
Our first main result substantially reduces the power of the logarithm in Burthe's mean-value estimate, bringing it much closer to the conjectured asymptotic of Bach and Huelsbergen.

\begin{theorem}\label{thm:main}
For every fixed real number $A>8/3$,
\[
    \sum_{n\leq x}G(n)\ll_A x(\log x)^A\qquad (x\geq2).
\]
\end{theorem}
In particular, one can take any fixed exponent larger than $8/3$. For instance
\[
    \sum_{n\leq x}G(n)
    \ll
    x(\log x)^3.
\]
Our second main result concerns the size of $G(n)$ for almost all
integers.
\begin{theorem}\label{thm:almost-all-new}
Let $c>4/e^2$ and $M>0$ be fixed. Then, for all sufficiently large $x$,
\[
 \#\left\{3\leq n\leq x:G(n)>c(\log n)^2\right\}\ll_{c,M}\frac{x}{(\log\log x)^M}.
\]
\end{theorem}
Since $4/e^2<1$, this immediately yields the following bound for almost all integers.
\begin{corollary}\label{cor:almost-all-new}
Unconditionally,
\[
    G(n)\leq(\log n)^2
\]
for almost all positive integers $n$.
\end{corollary}
Thus, unconditionally and for almost all integers $n$, we obtain the same logarithmic exponent $2$ as in Montgomery's pointwise bound under the Generalized Riemann Hypothesis.

Under GRH, the method of proof of Theorem~\ref{thm:almost-all-new}
yields a stronger exceptional-set estimate, as well as an upper bound
for the mean value of $G(n)$.
\begin{corollary}\label{cor:almost-all-GRH}
Assume the Generalized Riemann Hypothesis for Dirichlet $L$-functions. For every fixed $c>4/e^2$, there exists a constant $b_c>0$ such that,
for all sufficiently large $x$,
\[
 \#\{3\leq n\leq x:G(n)>c(\log n)^2\}\ll_c x\exp\left(-b_c\frac{\log x}{\log\log x}\right).
\]
Moreover,
\[
    \limsup_{x\to\infty}\frac{1}{x(\log x)^2}\sum_{3\leq n\leq x}G(n)\leq \frac4{e^2}.
\]
\end{corollary}

A complementary quantitative estimate is available when one considers
values larger than $(\log x)^2$.

\begin{corollary}\label{cor:quantitative-almost-all}
Let $A>2$. Then, for every fixed $0<\eta<4A$, one has
\[
\#\{n\leq x:G(n)>(\log x)^A\}\ll_{A,\eta}\frac{x}{(\log x)^\eta}.
\]
\end{corollary}

For prime moduli, we obtain a stronger exceptional-set estimate with a power saving and an additional exponential factor.

\begin{corollary}[Prime moduli]
\label{cor:prime-moduli-endpoint}
Let $A>2$, and let $p$ range over primes.
There exists a constant $c_A>0$ such that
\begin{equation}\label{eq:prime-moduli-endpoint}
\#\bigl\{p\le x:G(p)>(\log p)^A\bigr\}
\ll_A x^{1/(A-1)}
\exp\left(-c_A\frac{\log x}{\log\log x}\right).
\end{equation}
\end{corollary}
In particular, writing
$\delta=1-\frac{1}{A-1}>0$, we obtain
\[
\#\bigl\{p\le x:G(p)>(\log p)^A\bigr\}
\ll_A x^{1-\delta}.
\]
The same estimates hold with $(\log x)^A$
in place of $(\log p)^A$.

\begin{corollary}[Least prime primitive roots]
\label{cor:least-primitive-roots}
Let $g^*(p)$ denote the least prime primitive root modulo $p$,
and let $g(p)$ denote the least positive primitive root modulo $p$.
For every fixed positive integer $r$, every real number $A>5r$,
and every $\varepsilon>0$, one has
\begin{equation}\label{eq:least-prime-primitive-root-general}
\#\bigl\{p\le x:g^*(p)>(\log p)^A\bigr\}
\ll_{A,r,\varepsilon}
x^{\frac{5(r+1)}{2r(A-5r)}+\varepsilon}.
\end{equation}
\end{corollary}
In particular, for every $A>10$ and every $\varepsilon>0$,
\[
\#\bigl\{p\le x:g^*(p)>(\log p)^A\bigr\}
\ll_{A,\varepsilon}x^{\frac{5}{A-5}+\varepsilon}.
\]
Thus, for every $A>10$ and
\[
0<\delta<\frac{A-10}{A-5},
\]
one has
\[
\#\bigl\{p\le x:g^*(p)>(\log p)^A\bigr\}
\ll_{A,\delta}x^{1-\delta}.
\]
All these estimates also hold with $g(p)$ in place of $g^*(p)$,
and with $(\log x)^A$ in place of $(\log p)^A$.

We briefly describe the main ideas of the proofs.  The proof of
Theorem~\ref{thm:main} has two ingredients.  First, a high-moment
refinement of Burthe's argument, combined with Norton's character-sum
estimate, gives a saving in the primitive conductor which can be made
arbitrarily close to $3/4$.  Second, if $G(n)$ is large, then an
associated primitive character takes the value $1$ at almost all primes
in a suitable interval.  Raising the corresponding prime sum to a high
power and applying the multiplicative large sieve gives a dyadic bound
for the number of such exceptional primitive characters.  Balancing
the two estimates leads to the condition $A>8/3$.

For the almost-all result, we keep the more precise information in
the high-moment large-sieve estimate instead of replacing it by a
less precise power bound in the conductor. This allows us to reach the
endpoint logarithmic exponent $2$. For exponents $A>2$, a three-range
decomposition of the primitive conductors, combining the
Norton/high-moment estimate with the multiplicative large sieve,
gives the quantitative bound in
Corollary~\ref{cor:quantitative-almost-all}.

For prime moduli, stronger exceptional-set estimates are available.
Pollack~\cite{Pollack2012} proved that
\[
    \sum_{p\leq x}G(p)
    \ll \pi(x)(\log x)^2.
\]
Since every nonprincipal character modulo a prime is primitive,
one can associate to each exceptional prime a primitive character
whose conductor is that prime. For distinct exceptional primes,
the Pólya--Vinogradov inequality provides bounds for the pairwise
correlations of the associated characters. These bounds yield
a large-sieve-type inequality for the selected family of
exceptional characters. Combined with the high-moment argument
for prime sums, this gives the endpoint exponent $1/(A-1)$
and the additional exponential saving in
Corollary~\ref{cor:prime-moduli-endpoint}.
For general $n$, however, the character attached to $G(n)$
may be induced from a primitive character of a proper divisor
$d$ of $n$, and one must also control its contribution over
the multiples of $d$. This additional step is handled by
the Norton/high-moment argument and is the source of the
extra loss in the exponent in the present argument.
Related power-saving exceptional-set estimates for short
character sums over primes were obtained by
Garaev~\cite{Garaev2006}.

It is also interesting to compare these results with classical
bounds for the least primitive root modulo a prime. Burgess and
Elliott~\cite{BurgessElliott1968} showed that
\[
    \frac{1}{\pi(x)}\sum_{p\leq x}g(p)
    \ll (\log x)^2(\log\log x)^4{\color{blue}.}
\]
Martin~\cite[Thm.~1]{Martin1997} obtained bounds for the
least prime primitive root $g^*(p)$ using the shifted sieve.
In particular, his result, together with the standard estimate
$\omega(p-1)\ll\log p/\log\log p$, yields
\[
\#\{p\le x:g^*(p)>(\log p)^A\}
\ll_{A,\sigma}x^\sigma
\]
whenever $0<\sigma\le 20/21$ and $A>9/\sigma+15/4$.
Thus, a power saving follows for $A>66/5$.
Our Corollary~\ref{cor:least-primitive-roots} improves this
in two aspects. First, it gives a power-saving
exceptional-set estimate already for every $A>10$.
Second, it provides the smaller exponent
$5/(A-5)+\varepsilon$, which improves the exponent
obtainable from Martin's theorem for $A>66/5$.
Moreover, the additional parameter $r$ yields a further
improvement for sufficiently large $A$.
The same bounds hold for the least primitive root $g(p)$,
since $g(p)\le g^*(p)$.

\section{Notation and background}

For a nonprincipal Dirichlet character $\chi$, define
\[
 B(\chi)\coloneq 
 \min\{a\geq1:\chi(a)\notin\{0,1\}\}.
\]
If $\chi$ is principal, put $B(\chi)=0$.
If $n$ is a positive integer, write $\chi_{0,n}$ for the principal character
modulo $n$. We use the following estimate of Burthe.

\begin{lemma}[\cite{Burthe1997}, Lemma 3.1]\label{lem:BurtheLift}
Let $n\ge 2$ be an integer.
Let $\psi$ be a primitive character modulo $d$, with $d$ a divisor of $n$.
Then
\begin{equation}\label{eq:BurtheLift}
    B(\psi\chi_{0,n})\ll d(\log n)^2,
\end{equation}
with an absolute implied constant.
\end{lemma}
We also use Norton's character-sum estimate in a form for characters of
prime order.  \begin{lemma}[\cite{Norton1998}, Thm.\,1.6]\label{norton}
    Let $\psi$ be a nonprincipal Dirichlet character modulo $d$ such that
$\psi^k$ is principal. Then, for every positive integer $r$ and every $\eps>0$,
\begin{equation}\label{eq:Norton-general}
 \left|\sum_{M<m\leq M+H}\psi(m)\right|
 \ll_{\eps,r}
 R_k(d)^{1/r}
 H^{1-1/r}
 d^{(r+1)/(4r^2)+\eps},
\end{equation}
where
\[
 R_k(d)=\min\{M(d)^{3/4},Q(k)^{9/8}\},
\]
with
\[
M(n)  =\prod_{p^a \| n, a \geq 3} p^a, \qquad
Q(k)  =\prod_{p^a \| k, a \geq 2} p^a .\]
\end{lemma}
In particular, if $k$ is squarefree then
$Q(k)=1$, so $R_k(d)=1$.  Therefore, if $\psi$ has prime order,
then for every fixed $r\geq1$,
\begin{equation}\label{eq:Norton-prime}
 \left|\sum_{M<m\leq M+H}\psi(m)\right|
 \ll_{\eps,r}
 H^{1-1/r}
 d^{(r+1)/(4r^2)+\eps},
\end{equation}
with no cubefree hypothesis on $d$.\\
Finally, we use the standard multiplicative large sieve.
\begin{lemma}[{\cite[Thm.~7.13]{IwaniecKowalski}}]
\label{lem:large-sieve}
Let $Q,N\geq1$, and let $(a_m)$ be a finite sequence of complex numbers
supported on the integers $1\leq m\leq N$.  Then
\begin{equation}\label{eq:large-sieve}
 \sum_{q\leq Q}\frac{q}{\varphi(q)}
 \sum_{\chi\bmod q}^{*}
 \left|\sum_{m\leq N}a_m\chi(m)\right|^2
 \leq
 (Q^2+N-1)\sum_{m\leq N}|a_m|^2,
\end{equation}
where the $*$ indicates that the summation is over primitive characters only.
\end{lemma}

\section{Prime-order characters}

Recall that $G(n)$ is the least positive integer $G$ such that the
classes of all $a\leq G$ with $(a,n)=1$ generate $ \Gamma_n\coloneq (\Z/n\Z)^\times$.
\begin{lemma}\label{lem:prime-order}
Suppose that $\Gamma_n$ is nontrivial.  There exists a nonprincipal
Dirichlet character $\chi_n\bmod n$ of prime order such that
\[
        B(\chi_n)=G(n).
\]
\end{lemma}

\begin{proof}
Put $G=G(n)$ and let $H$ be the subgroup of $\Gamma_n$ generated by
the residue classes of all positive integers $a<G$ with $(a,n)=1$.
By the minimality of $G$, one has
\[
        H\neq\Gamma_n,
        \qquad
        \Gamma_n=\langle H,\overline G\rangle,
\]
where we write $\overline a$ for the residue class of $a$ in
$(\mathbb Z/n\mathbb Z)^\times$.
Notice also that $(G,n)=1$, otherwise the set of integers $\leq G$
which are coprime to $n$ would be the same as the corresponding set
with $G$ replaced by $G-1$, contradicting the minimality of $G$.
It follows that the quotient $\Gamma_n/H$ is a nontrivial cyclic
group, generated by $\overline G H$.  Write
\[
        |\Gamma_n/H|=m>1
\]
and choose a prime $p$ dividing $m$.  If $g\coloneq \overline G H$ denotes the
generator of this cyclic quotient, define
\[
        \widetilde\chi(g)\coloneq e^{2\pi i/p}.
\]
This is well defined because
$p$ divides $m$, and it defines a character of
$\Gamma_n/H$ of exact order $p$.  Let
\[
\pi:\Gamma_n\to \Gamma_n/H
\]
be the quotient map, and define
\[
\chi_n\coloneq \widetilde\chi\circ\pi.
\]
Then $\chi_n$ is a character of $\Gamma_n$ of order $p$,
hence a Dirichlet character modulo $n$ of order $p$.
If $1\leq a<G$, then either $(a,n)>1$, in which case $\chi_n(a)=0$,
or $(a,n)=1$, in which case $\overline a\in H$ and $\chi_n(a)=1$.
On the other hand,
\[
        \chi_n(G)=e^{2\pi i/p}\notin\{0,1\}.
\]
Therefore $B(\chi_n)=G$.
\end{proof}

\begin{lemma}
\label{lem:order-inducer}
Let $\chi\bmod n$ be induced by the primitive character $\psi\bmod d$,
where $d$ divides $n$.  Then
\[
        \ord(\chi)=\ord(\psi).
\]
\end{lemma}

\begin{proof}
The values of $\chi$ on integers coprime to $n$ are values of $\psi$,
so $\ord(\chi)$ divides $\ord(\psi)$.
Conversely, let $b$ be any residue class modulo $d$ with $(b,d)=1$.
We claim that there exists $a\equiv b\pmod d$ with $(a,n)=1$.
Indeed, for each prime $\ell\mid n$ with $\ell\nmid d$, impose in
addition the congruence $a\equiv1\pmod\ell$.  These congruences are
compatible with $a\equiv b\pmod d$ and the Chinese remainder theorem
produces such an $a$.  For primes dividing $d$, coprimality follows
from $(b,d)=1$.
Hence every nonzero value of $\psi$ occurs among the nonzero values
of $\chi$.  Their images as finite subgroups of $\C^\times$ are
therefore equal, and the two characters have the same order.
\end{proof}
Let $\psi\bmod d$ be the primitive character inducing $\chi_n$.
By Lemma~\ref{lem:order-inducer},
\[
    \ord(\psi)=\ord(\chi_n)=p,
\]
where $p$ is prime. Hence $\psi^p$ is principal and $p$ is
squarefree. Therefore Norton's factor $R_p(d)$ equals $1$, and
Lemma~\ref{norton} gives, for every fixed integer $r\ge1$,
\[
    \left|\sum_{M<m\le M+H}\psi(m)\right|
    \ll_{\varepsilon,r}
    H^{1-1/r}d^{(r+1)/(4r^2)+\varepsilon}.
\]
In particular, no cubefree hypothesis on $d$ is required.

\section{A high-moment exceptional-set estimate}\label{sec:highmoment}

Fix an integer $r\geq2$ and set
\[
        a_r\coloneq 1-\frac1r,
        \qquad
        b_r\coloneq \frac{r+1}{4r^2}.
\]
For $N\geq1$, define
\begin{equation}\label{eq:Fr}
        F_r(N)\coloneq \sum_{g\mid N}g^{-a_r}
        =\sum_{g\mid N}g^{-1+1/r}.
\end{equation}

\begin{lemma}\label{lem:moments}
For every fixed pair of integers $r\geq2$ and $k\geq1$,
\[
        \sum_{N\leq y}F_r(N)^k\ll_{r,k} y
        \qquad (y\geq1).
\]
\end{lemma}

\begin{proof}
Expanding the $k$-th power and interchanging summations gives
\begin{align*}
 \sum_{N\leq y}F_r(N)^k
 &=
 \sum_{g_1,\ldots,g_k\geq1}
 (g_1\cdots g_k)^{-a_r}
 \#\{N\leq y:[g_1,\ldots,g_k]\mid N\}\leq
 y
 \sum_{g_1,\ldots,g_k\geq1}
 \frac{(g_1\cdots g_k)^{-a_r}}
      {[g_1,\ldots,g_k]}.
\end{align*}
The multiple series is multiplicative.  Its Euler factor at a prime
$q$ is
\[
 \mathcal E_q
 =
 \sum_{e_1,\ldots,e_k\geq0}
 q^{-a_r(e_1+\cdots+e_k)-\max(e_1,\ldots,e_k)}.
\]
For a nonzero $k$-tuple put $E=\max(e_1,\ldots,e_k)\geq1$.  Since
$e_1+\cdots+e_k\geq E$,
\[
 a_r(e_1+\cdots+e_k)+E
 \geq (a_r+1)E
 =\left(2-\frac1r\right)E.
\]
There are at most $(E+1)^k$ tuples with maximum $E$.  Hence
\[
 \mathcal E_q-1
 \ll_{r,k}
 \sum_{E\geq1}(E+1)^k
 q^{-(2-1/r)E}
 \ll_{r,k}q^{-(2-1/r)}.
\]
Since $2-1/r>1$, the Euler product
$\prod_q\mathcal E_q$ converges absolutely.  The asserted estimate
follows.
\end{proof}

The next proposition is a flexible version of the exceptional-set estimate
used by Burthe.  The point is that the truncation parameter $H$ is left free for now.
We will later choose it as a function of the conductor.
\begin{proposition}\label{prop:flexible-exceptional}
There is an absolute constant $C_0>0$ with the following property.  Fix
integers $r\geq2$ and $k\geq1$, and let $\psi$ be a primitive character of
prime order and conductor $d$.  Let $x\geq d$, and let $H$ satisfy
\begin{equation}\label{eq:H-range}
        C_0\log x\leq H\leq x.
\end{equation}
Then, for every $\eta>0$,
\begin{equation}\label{eq:flex-bad-count}
 \#\left\{
 n\leq x:\ d\mid n,\ B(\psi\chi_{0,n})>H
 \right\}
 \ll_{r,k,\eta}
 \frac{x}{d}
 (\log H)^k H^{-k/r}d^{kb_r+\eta}.
\end{equation}
Consequently,
\begin{equation}\label{eq:flex-fixedpsi-sum}
 \sum_{\substack{n\leq x\\d\mid n}}
 B(\psi\chi_{0,n})
 \ll_{r,k,\eta}
 \frac{xH}{d}
 +x(\log x)^2(\log H)^kH^{-k/r}d^{kb_r+\eta}.
\end{equation}
\end{proposition}

\begin{proof}
Write $n=dN$.  Suppose first that
\[
        B(\psi\chi_{0,n})>H.
\]
If $m\leq H$ and $(m,n)=1$, then
\[
        \psi(m)=\psi(m)\chi_{0,n}(m)=1.
\]
Consider
\[
        S(N)\coloneq \sum_{\substack{m\leq H\\(m,N)=1}}\psi(m).
\]
If $(m,N)=1$ but $(m,d)>1$, then $\psi(m)=0$.  If $(m,N)=1$ and
$(m,d)=1$, then $(m,n)=1$ and $\psi(m)=1$.  Therefore
\begin{equation}\label{eq:S-count}
        S(N)=\#\{m\leq H:(m,n)=1\}.
\end{equation}
In particular, using only primes in $(H/2,H]$,
\begin{equation}\label{eq:S-primes}
 S(N)\geq
 \pi(H)-\pi(H/2)
 -\#\{p\mid n:H/2<p\leq H\}.
\end{equation}
For all sufficiently large $H$, the prime number theorem gives an absolute
constant $c>0$ such that
\begin{equation}\label{eq:short-prime-lower}
        \pi(H)-\pi(H/2)\geq c\frac{H}{\log H}.
\end{equation}
If $R$ denotes the number of prime divisors of $n$ in $(H/2,H]$, then
\[
        (H/2)^R\leq n\leq x,
\]
and hence
\begin{equation}\label{eq:large-prime-divisors}
        R\leq\frac{\log x}{\log(H/2)}.
\end{equation}
For $x$ sufficiently large, \eqref{eq:H-range} implies
$\log(H/2)\geq\tfrac12\log H$.  Choosing $C_0$ sufficiently large in terms
of the absolute constant $c$ in \eqref{eq:short-prime-lower}, we obtain from
\eqref{eq:S-primes}, \eqref{eq:short-prime-lower}, and \eqref{eq:large-prime-divisors} that
\begin{equation}\label{eq:S-lower-flex}
        S(N)\gg\frac{H}{\log H}.
\end{equation}
On the other hand, M\"obius inversion gives
\begin{align}
 S(N)
 &=
 \sum_{m\leq H}\psi(m)
 \sum_{\substack{g\mid m\\g\mid N}}\mu(g)=
 \sum_{g\mid N}\mu(g)\psi(g)
 \sum_{h\leq H/g}\psi(h).
 \label{eq:mobius-flex}
\end{align}
Terms with $H/g<1$ are empty.  By Lemma~\ref{norton}, for every
$\eps_1>0$,
\[
 \left|\sum_{h\leq H/g}\psi(h)\right|
 \ll_{\eps_1,r}
 (H/g)^{1-1/r}d^{b_r+\eps_1}.
\]
Hence
\begin{equation}\label{eq:S-upper-flex}
 |S(N)|
 \ll_{\eps_1,r}
 H^{1-1/r}d^{b_r+\eps_1}F_r(N).
\end{equation}
Combining \eqref{eq:S-lower-flex} and \eqref{eq:S-upper-flex}, we obtain
\begin{equation}\label{eq:Fr-lower-flex}
 F_r(N)
 \gg_{\eps_1,r}
 \frac{H^{1/r}d^{-b_r-\eps_1}}{\log H}.
\end{equation}
Let $\mathcal E$ be the set of $N\leq x/d$ arising from integers $n=dN$
for which $B(\psi\chi_{0,n})>H$.  Raising
\eqref{eq:Fr-lower-flex} to the $k$-th power and applying
Lemma~\ref{lem:moments}, we get
\[
 \#\mathcal E
 \ll_{r,k,\eps_1}
 \frac{x}{d}(\log H)^kH^{-k/r}d^{kb_r+k\eps_1}.
\]
Taking $\eps_1=\eta/k$ proves \eqref{eq:flex-bad-count}.\\
To prove \eqref{eq:flex-fixedpsi-sum}, split the multiples $n\leq x$ of $d$
according as $B(\psi\chi_{0,n})\leq H$ or $B(\psi\chi_{0,n})>H$.  The first
part contributes at most
\[
        \frac{xH}{d}.
\]
For the exceptional part, Lemma~\ref{lem:BurtheLift} and
\eqref{eq:flex-bad-count} give
\begin{align*}
 &\sum_{\substack{n\leq x,\ d\mid n\\
                   B(\psi\chi_{0,n})>H}}
 B(\psi\chi_{0,n})\ll
 d(\log x)^2
 \cdot
 \frac{x}{d}(\log H)^kH^{-k/r}d^{kb_r+\eta},
\end{align*}
which is the second term in \eqref{eq:flex-fixedpsi-sum}.
\end{proof}
We will use the following elementary optimization.
\begin{lemma}[Optimization]\label{lem:optimization}
For fixed integers $r\geq2$ and $k\geq1$, define
\begin{equation}\label{eq:theta-opt}
 \theta_{r,k}
 \coloneq \frac{4r^2+k(r+1)}{4r(r+k)}
\end{equation}
and
\begin{equation}\label{eq:lambda-opt}
 \lambda_{r,k}
 \coloneq \frac{k(3r-1)}{4r(r+k)}.
\end{equation}
Then
\begin{equation}\label{eq:balance-identity}
 1-\theta_{r,k}
 =k\left(\frac{\theta_{r,k}}{r}-b_r\right)
 =\lambda_{r,k}.
\end{equation}
Moreover,
\[
        \sup_{r\geq2,\ k\geq1}\lambda_{r,k}=\frac34.
\]
\end{lemma}

\begin{proof}
Solving
\[
 1-\theta
 =k\left(\frac{\theta}{r}-\frac{r+1}{4r^2}\right)
\]
for $\theta$ gives \eqref{eq:theta-opt}, and subtracting from $1$ gives
\eqref{eq:lambda-opt}.  For fixed $r$,
\[
 \lim_{k\to\infty}\lambda_{r,k}
 =\frac{3r-1}{4r}
 =\frac34-\frac1{4r}.
\]
Letting $r\to\infty$ proves the assertion.
\end{proof}

\section{Large prime sums and the multiplicative large sieve}
\label{sec:large-sieve}
For $Y\geq 3$, write
\[
    \mathcal P(Y)\coloneq \{p\text{ prime}:Y/2<p\leq Y\},
    \qquad
    P(Y)\coloneq \#\mathcal P(Y).
\]

\begin{lemma}\label{lem:large-G-prime-sum}
Fix $A>1$, put
\[
    L\coloneq \log x,
    \qquad
    Y\coloneq L^A,
\]
and suppose that $x$ is sufficiently large in terms of $A$.
Let $n\leq x$ satisfy $G(n)>Y$.  Let $\chi_n\bmod n$ be the
prime-order character supplied by Lemma~\ref{lem:prime-order}, and let
$\psi\bmod d$ be the primitive character inducing $\chi_n$.  Then:
\begin{enumerate}
\item $\ord(\psi)$ is prime;
\item there is an absolute constant $C_B>0$ such that
\begin{equation}\label{eq:Dx-new}
    d\geq D_x
    \coloneq \max\left\{1,C_B^{-1}L^{A-2}\right\}.
\end{equation}
In particular, if $A>2$, then for sufficiently large $x$, $ D_x=C_B^{-1}L^{A-2}$;
\item one has
\begin{equation}\label{eq:prime-sum-large}
    \left|
        \sum_{p\in\mathcal P(Y)}\psi(p)
    \right|
    \geq \frac12 P(Y).
\end{equation}
\end{enumerate}
\end{lemma}

\begin{proof}
By Lemma~\ref{lem:order-inducer},
\[
    \ord(\psi)=\ord(\chi_n),
\]
and the latter is prime by Lemma~\ref{lem:prime-order}.  This proves
the first assertion.\\
Since $\psi\bmod d$ induces $\chi_n\bmod n$, we have
\[
    \chi_n=\psi\chi_{0,n}.
\]
Hence Lemma~\ref{lem:prime-order} gives
\[
    G(n)=B(\chi_n)=B(\psi\chi_{0,n}).
\]
Choose an absolute constant $C_B>0$ such that
Lemma~\ref{lem:BurtheLift} gives
\[
    B(\psi\chi_{0,n})
    \leq C_Bd(\log n)^2.
\]
Since $n\leq x$ and $G(n)>Y$, it follows that
\[
    Y
    <G(n)
    =B(\psi\chi_{0,n})
    \leq C_Bd(\log n)^2
    \leq C_BdL^2.
\]
Therefore
\[
    d>\frac{Y}{C_BL^2}
    =C_B^{-1}L^{A-2}.
\]
Since also $d\geq1$, we obtain
\[
    d\geq
    \max\left\{1,C_B^{-1}L^{A-2}\right\}
    =D_x,
\]
which proves \eqref{eq:Dx-new}.\\
It remains to prove \eqref{eq:prime-sum-large}.  Let
$p\in\mathcal P(Y)$ and suppose that $p\nmid n$.  Since
\[
    p\leq Y<G(n)=B(\chi_n),
\]
the definition of $B(\chi_n)$ gives
\[
    \chi_n(p)\in\{0,1\}.
\]
As $p\nmid n$, one has $\chi_n(p)\neq0$, and hence
\[
    \chi_n(p)=1.
\]
Moreover $\chi_{0,n}(p)=1$, so from
$\chi_n=\psi\chi_{0,n}$ we obtain
\[
    \psi(p)=1.
\]
Thus every prime $p\in\mathcal P(Y)$ which fails to contribute $1$ to
the sum must divide $n$.
Put
\[
    R_n
    \coloneq \#\{p\in\mathcal P(Y):p\mid n\}.
\]
Since every such prime satisfies $p>Y/2$, we have
\[
    \left(\frac{Y}{2}\right)^{R_n}
    \leq \prod_{\substack{p\mid n\\ p\in\mathcal P(Y)}}p
    \leq n
    \leq x.
\]
Consequently,
\begin{equation}\label{eq:Rn-bound}
    R_n
    \leq
    \frac{L}{\log(Y/2)}.
\end{equation}
For primes $p\nmid n$ in $\mathcal P(Y)$ we have $\psi(p)=1$, while
for the remaining primes $|\psi(p)|\leq1$.  Therefore
\[
\begin{aligned}
    \left|
        \sum_{p\in\mathcal P(Y)}\psi(p)-P(Y)
    \right|
    &=
    \left|
        \sum_{\substack{p\in\mathcal P(Y)\\p\mid n}}
        \bigl(\psi(p)-1\bigr)
    \right| \leq
    2R_n.
\end{aligned}
\]
It follows that
\begin{equation}\label{eq:prime-sum-preliminary}
    \left|
        \sum_{p\in\mathcal P(Y)}\psi(p)
    \right|
    \geq P(Y)-2R_n.
\end{equation}
By the prime number theorem,
\begin{equation}\label{eq:PY-asymp}
    P(Y)
    =\pi(Y)-\pi(Y/2)
    \asymp \frac{Y}{\log Y}.
\end{equation}
On the other hand, since $Y=L^A$,
\[
    \frac{R_n}{P(Y)}
    \ll
    \frac{L}{\log(Y/2)}
    \frac{\log Y}{Y}
    \ll_A
    L^{1-A}.
\]
Since $A>1$, this tends to zero as $x\to\infty$.  Thus, for
sufficiently large $x$,
\[
    2R_n\leq \frac12 P(Y).
\]
Combining this with \eqref{eq:prime-sum-preliminary} gives
\[
    \left|
        \sum_{p\in\mathcal P(Y)}\psi(p)
    \right|
    \geq \frac12P(Y),
\]
as required.
\end{proof}
The next lemma is the large-sieve estimate that replaces the zero-density
argument.
\begin{lemma}\label{lem:large-sieve-exceptional}
Fix $A>2$ and a real number $\rho>2/A$.  Put $L=\log x$ and $Y=L^A$.
Let $c>0$ be fixed.  For a real number $Q$ satisfying
\begin{equation}\label{eq:Q-range-LS}
        cL^{A-2}\leq Q\leq x,
\end{equation}
let $R(Q)$ denote the number of primitive Dirichlet characters $\psi$ of
conductor $d$ with $Q\leq d<2Q$ for which
\begin{equation}\label{eq:LS-exceptional-condition}
 \left|\sum_{p\in\mathcal P(Y)}\psi(p)\right|
 \geq\frac12P(Y).
\end{equation}
Then, for all sufficiently large $x$ in terms of $A,\rho,c$,
\begin{equation}\label{eq:RQ-bound}
        R(Q)\ll_{A,\rho,c}Q^\rho,
\end{equation}
uniformly for $Q$ in the range \eqref{eq:Q-range-LS}.
\end{lemma}

\begin{proof}
Write
\[
        P\coloneq P(Y),
        \qquad
        V\coloneq \log Y=A\log L,
        \qquad
        U\coloneq \log(2Q).
\]
Choose
\begin{equation}\label{eq:h-choice}
        h\coloneq \max\left\{1,
        \left\lceil\frac{2U}{V}\right\rceil\right\}.
\end{equation}
Then
\begin{equation}\label{eq:Yh-Q2}
        Y^h\geq(2Q)^2.
\end{equation}
Expand
\begin{equation}\label{eq:power-prime-sum}
 \left(\sum_{p\in\mathcal P(Y)}\psi(p)\right)^h
 =\sum_{m\leq Y^h}a_h(m)\psi(m),
\end{equation}
where $a_h(m)$ is the number of ordered $h$-tuples
$(p_1,\ldots,p_h)\in\mathcal P(Y)^h$ for which $p_1\cdots p_h=m$.
Unique factorization implies
\begin{equation}\label{eq:ah-second-moment}
        \sum_m a_h(m)^2\leq h!P^h.
\end{equation}
Indeed, once one ordered $h$-tuple is fixed, there are at most $h!$
ordered tuples with the same product.  Summing over all $P^h$ initial tuples
proves \eqref{eq:ah-second-moment}.\\
Every character counted by $R(Q)$ satisfies
\[
 \left|\sum_{m\leq Y^h}a_h(m)\psi(m)\right|^2
 \geq(P/2)^{2h}.
\]
Since $d/\varphi(d)\geq1$, Lemma~\ref{lem:large-sieve}, applied with modulus
bound $2Q$ and length $Y^h$, gives
\begin{align*}
 R(Q)(P/2)^{2h}
 &\leq
 \sum_{d<2Q}\frac{d}{\varphi(d)}
 \sum_{\psi\bmod d}^{*}
 \left|\sum_{m\leq Y^h}a_h(m)\psi(m)\right|^2\leq
 \bigl((2Q)^2+Y^h-1\bigr)
 \sum_m a_h(m)^2\leq
 2Y^h h!P^h,
\end{align*}
where the last inequality uses \eqref{eq:Yh-Q2} and
\eqref{eq:ah-second-moment}.  Hence
\begin{equation}\label{eq:RQ-preuniform}
 R(Q)
 \leq
 2\cdot4^h h!\left(\frac{Y}{P}\right)^h.
\end{equation}
By the prime number theorem, $P\gg Y/\log Y$.  Since $h!\leq h^h$, we obtain
\begin{equation}\label{eq:RQ-ChV}
        R(Q)\leq (C hV)^h
\end{equation}
for an absolute constant $C>0$, after enlarging $C$ if necessary.
It remains to verify uniformly that the right side of
\eqref{eq:RQ-ChV} is at most $Q^\rho$.  From \eqref{eq:h-choice},
\begin{equation}\label{eq:h-upper}
        h\leq\frac{2U}{V}+1,
        \qquad
        hV\leq2U+V.
\end{equation}
Also, by \eqref{eq:Q-range-LS},
\begin{equation}\label{eq:U-lower}
        U\geq(A-2)\log L+O_{A,c}(1),
\end{equation}
and $U\leq L+O(1)$ because $Q\leq x$.
Choose a fixed $\delta>0$, to be specified below.  First suppose that
$U\geq L^\delta$.  From \eqref{eq:h-upper},
\[
 \log(ChV)\leq\log\bigl(C(2U+V)\bigr)
 \leq\log L+O_{A,\delta}(1).
\]
Since $V=A\log L$, it follows that
\begin{align*}
 \frac{\log R(Q)}{\log Q}
 &\leq
 \frac{\left(2U/V+1\right)
       \left(\log L+O_{A,\delta}(1)\right)}{\log Q}\leq \frac{2}{A}+o_{A,\delta}(1),
\end{align*}
where we used $U/\log Q=1+o(1)$ and
$(\log L)/U\leq(\log L)/L^\delta=o(1)$.  Since $\rho>2/A$, this is less than
$\rho$ for sufficiently large $x$.
Now suppose that $U<L^\delta$.  From \eqref{eq:h-upper},
\[
 \log(ChV)
 \leq\log\bigl(C(2U+V)\bigr)
 \leq \delta\log L+O_{A,\delta}(1).
\]
Using \eqref{eq:U-lower}, we get
\begin{align}\label{eq:36}
 \frac{\log R(Q)}{\log Q}
 &\leq
 \frac{\left(2U/V+1\right)
       \left(\delta\log L+O_{A,\delta}(1)\right)}{\log Q}\leq
 \frac{2\delta}{A}
 +\frac{\delta}{A-2}
 +o_{A,c,\delta}(1).
\end{align}
Choose $\delta>0$ so small that
\[
        \frac{2\delta}{A}+\frac{\delta}{A-2}<\rho.
\]
Then \eqref{eq:36} is also less than $\rho$ for sufficiently large $x$.
This proves \eqref{eq:RQ-bound} uniformly throughout
\eqref{eq:Q-range-LS}.
\end{proof}

\section{Proof of Theorem~\ref{thm:main}}
It is enough to consider sufficiently large $x$.  Fix $A>8/3$ and put
\[
        L\coloneq \log x,
        \qquad
        Y\coloneq L^A.
\]
Since $2/A<3/4$, choose a real number $\rho$ such that
\begin{equation}\label{eq:rho-choice-main}
        \frac{2}{A}<\rho<\frac34.
\end{equation}
By Lemma~\ref{lem:optimization}, we may choose fixed integers $r\geq2$ and
$k\geq1$ such that, writing
\[
        \theta\coloneq \theta_{r,k},
        \qquad
        \lambda\coloneq \lambda_{r,k},
\]
one has
\begin{equation}\label{eq:lambda-rho-main}
        \rho<\lambda<\frac34.
\end{equation}
In particular,
\begin{equation}\label{eq:balance-main}
        1-\theta
        =k\left(\frac{\theta}{r}-b_r\right)
        =\lambda.
\end{equation}
Choose $\eta>0$ sufficiently small that
\begin{equation}\label{eq:eta-choice-main}
        \rho+3\eta<\lambda.
\end{equation}
Split
\begin{equation}\label{eq:main-split}
        \sum_{n\leq x}G(n)=S_{\rm small}+S_{\rm large},
\end{equation}
where $S_{\rm small}$ is the contribution of the integers with $G(n)\leq Y$
and $S_{\rm large}$ is the contribution of the integers with $G(n)>Y$.
Trivially,
\begin{equation}\label{eq:Ssmall}
        S_{\rm small}\leq xY=xL^A.
\end{equation}
For each $n$ counted by $S_{\rm large}$, choose the prime-order character
$\chi_n$ from Lemma~\ref{lem:prime-order}, and let $\psi_n\bmod d_n$ be its
primitive inducing character.  By Lemma~\ref{lem:large-G-prime-sum},
\begin{equation}\label{eq:associated-new}
        d_n\geq D_x\coloneq C_B^{-1}L^{A-2},
        \qquad
        \ord(\psi_n)\text{ is prime},
\end{equation}
and
\begin{equation}\label{eq:associated-primesum}
 \left|\sum_{p\in\mathcal P(Y)}\psi_n(p)\right|
 \geq\frac12P(Y).
\end{equation}
For each $d\geq1$, let $\mathcal U_d(x)$ be the set of primitive characters
$\psi\bmod d$ of prime order satisfying
\eqref{eq:associated-primesum}, and write
\[
        u_d(x)\coloneq \#\mathcal U_d(x).
\]
Lemma~\ref{lem:large-sieve-exceptional}, with $c=C_B^{-1}$, gives, uniformly
for every dyadic $Q$ with $D_x\leq Q\leq x$,
\begin{equation}\label{eq:ud-dyadic}
        \sum_{Q\leq d<2Q}u_d(x)\ll_A Q^\rho.
\end{equation}
Here and below the implied constants may also depend on the fixed choices of
$\rho,r,k$ and $\eta$, all of which depend only on $A$.\\
Let $C_0$ be the constant from Proposition~\ref{prop:flexible-exceptional}
and, for $D_x\leq d\leq x$, put
\begin{equation}\label{eq:Hd-def}
        H_d\coloneq \max\{C_0L,d^\theta\}.
\end{equation}
For sufficiently large $x$, one has $C_0L\leq x$ and $d^\theta\leq x$, so
$H_d$ lies in the range required by Proposition~\ref{prop:flexible-exceptional}.
For every $\psi\in\mathcal U_d(x)$, that proposition gives
\begin{equation}\label{eq:Idpsi}
 \sum_{\substack{n\leq x\\d\mid n}}
 B(\psi\chi_{0,n})
 \ll_A
 \frac{xH_d}{d}
 +xL^2(\log H_d)^kH_d^{-k/r}d^{kb_r+\eta}.
\end{equation}
Every integer counted by $S_{\rm large}$ contributes to one of the pairs
$(d,\psi)$ on the right.  Since
$G(n)=B(\psi_n\chi_{0,n})$, enlarging the sum gives
\begin{equation}\label{eq:Slarge-enlarge}
 S_{\rm large}
 \leq
 \sum_{D_x\leq d\leq x}
 \sum_{\psi\in\mathcal U_d(x)}
 \sum_{\substack{n\leq x\\d\mid n}}
 B(\psi\chi_{0,n}).
\end{equation}
We estimate \eqref{eq:Slarge-enlarge} by splitting the conductors at
\begin{equation}\label{eq:D1}
        D_1\coloneq (C_0L)^{1/\theta}.
\end{equation}
When the dyadic decomposition has one interval meeting both sides of $D_1$,
we treat that single transition interval separately.  On such an interval
$d\asymp D_1$ and $H_d\asymp L\asymp d^\theta$, so either of the estimates
below applies after changing the implied constant.

\medskip
\noindent
\emph{Small-conductor range: $D_x\leq d\leq D_1$.}
Here $H_d=C_0L$.  On a dyadic interval $Q\leq d<2Q$ contained in this
range, \eqref{eq:Idpsi} and \eqref{eq:ud-dyadic} give a contribution
\begin{equation}\label{eq:small-block}
 \ll_A
 xLQ^{\rho-1}
 +xL^{2-k/r}(\log L)^kQ^{\rho+kb_r+\eta}.
\end{equation}
Since $\rho<1$, summing the first term in \eqref{eq:small-block} over dyadic
$Q\geq D_x$ gives
\[
        \ll_A xL D_x^{\rho-1}\ll_A xL.
\]
The exponent $\rho+kb_r+\eta$ is positive, so the dyadic sum of the second
term is dominated by its last interval.  Using \eqref{eq:D1} and
\eqref{eq:balance-main}, we obtain
\begin{align*}
 &xL^{2-k/r}(\log L)^kD_1^{\rho+kb_r+\eta}\ll_A
 x(\log L)^k
 L^{2-k/r+(\rho+kb_r+\eta)/\theta}=
 x(\log L)^k
 L^{2+(\rho+\eta-\lambda)/\theta}.
\end{align*}
By \eqref{eq:eta-choice-main}, the exponent of $L$ in the last line is
strictly smaller than $2$.  Since $k$ is fixed,
$(\log L)^k=L^{o(1)}$, and hence the total contribution of the
small-conductor range is
\begin{equation}\label{eq:small-conductor-total}
        \ll_A xL^2.
\end{equation}
If $D_x>D_1$, this range is empty and there is nothing to prove.

\medskip
\noindent
\emph{Large-conductor range: $d>D_1$.}
Here $H_d=d^\theta$.  By \eqref{eq:balance-main},
\[
        \frac{xH_d}{d}=xd^{-\lambda}
\]
and
\[
 xL^2(\log H_d)^kH_d^{-k/r}d^{kb_r+\eta}
 \ll_A
 xL^2(\log d)^k d^{-\lambda+\eta}.
\]
Since $d>D_1\to\infty$, for sufficiently large $x$ we may absorb the fixed
power of $\log d$ into $d^\eta$.  Thus
\begin{equation}\label{eq:large-conductor-per-char}
 \sum_{\substack{n\leq x\\d\mid n}}
 B(\psi\chi_{0,n})
 \ll_A xL^2d^{-\lambda+2\eta}
\end{equation}
for every $\psi\in\mathcal U_d(x)$.  By \eqref{eq:ud-dyadic}, a dyadic
interval $Q\leq d<2Q$ contributes at most
\[
        \ll_A xL^2Q^{\rho-\lambda+2\eta}.
\]
The exponent is negative by \eqref{eq:eta-choice-main}.  Summing the resulting
geometric series over the dyadic intervals with $Q\gg D_1$ therefore gives
\begin{equation}\label{eq:large-conductor-total}
        \ll_A xL^2.
\end{equation}
Combining \eqref{eq:Slarge-enlarge},
\eqref{eq:small-conductor-total}, and
\eqref{eq:large-conductor-total}, we conclude that
\begin{equation}\label{eq:Slarge-final}
        S_{\rm large}\ll_A xL^2.
\end{equation}
Together with \eqref{eq:Ssmall}, and using $A>8/3>2$, this proves
\[
        \sum_{n\leq x}G(n)
        \ll_A xL^A.
\]
The finitely many smaller values of $x$ are absorbed into the implied
constant.
\begin{remark}[The barrier of the method]\label{rem:barrier}
Lemma~\ref{lem:large-sieve-exceptional} gives a dyadic exceptional-character
exponent arbitrarily close to $2/A$, while the Norton/high-moment argument in
Section~\ref{sec:highmoment} gives a conductor saving arbitrarily close to
$3/4$.  The proof requires a strict gap between these exponents.  Thus the
present argument works precisely in the range
\[
        \frac{2}{A}<\frac34,
\]
that is, $A>8/3$.  At $A=8/3$ both limiting exponents are $3/4$, so the
argument does not reach the endpoint.
\end{remark}

\section{Proof of Theorem \ref{thm:almost-all-new}}
The proof is divided into two preliminary steps.  We first use the
Siegel--Walfisz theorem to show that characters of small conductor
cannot satisfy the large prime-sum condition which arises when
$G(n)$ is large.  We then use the multiplicative large sieve, with a
more precise choice of the high moment, to obtain a weighted estimate
for the remaining exceptional primitive characters.  Finally, we
associate such a primitive character to every integer $n$ for which
$G(n)$ is large and use the fact that its conductor divides $n$ to
deduce the theorem.
\begin{lemma}\label{lem:endpoint-small-conductors}
For every fixed $B>0$, uniformly over nonprincipal Dirichlet characters $\psi\bmod d$ with $d\le (\log z)^B$, one has
\begin{equation}\label{eq:endpoint-SW}
    \left|\sum_{p\leq z}\psi(p)\right|
    \ll_B
    \frac{z}{(\log z)^3}
    \qquad (z\to\infty).
\end{equation}
\end{lemma}
\begin{proof}
By the Siegel--Walfisz theorem \cite[Cor.\,5.29]{IwaniecKowalski}, applied with logarithmic exponent $B+3$, uniformly for $d\le (\log z)^B$ and $(a,d)=1$, we have
\[
    \pi(z;d,a)
    =
    \frac{\operatorname{Li}(z)}{\varphi(d)}
    +
    O_B\left(
        \frac{z}{(\log z)^{B+3}}
    \right),
\]
where $\pi(z;d,a)$ denotes the number of primes $p\leq z$ satisfying
$p\equiv a\pmod d$.\\
Since $\psi(p)=0$ when $p|d$, we can group the remaining primes according to their reduce residue class modulo $d$:
\[
    \sum_{p\leq z}\psi(p)
    =
    \sum_{\substack{a\bmod d\\(a,d)=1}}
    \psi(a)\pi(z;d,a).
\]
Substituting the Siegel--Walfisz estimate gives
\begin{align*}
    \sum_{p\leq z}\psi(p)
    &=
    \frac{\operatorname{Li}(z)}{\varphi(d)}\sum_{\substack{a\bmod d\\(a,d)=1}}\psi(a)
    +
    O_B\left(\frac{z}{(\log z)^{B+3}} \sum_{\substack{a\bmod d\\(a,d)=1}} |\psi(a)|\right).
\end{align*}
Since $\psi$ is nonprincipal, 
\[
    \sum_{\substack{a\bmod d\\(a,d)=1}}\psi(a)=0,
\]
so the main term vanishes.  Moreover $|\psi(a)|=1$ on the residue classes.  Hence
\[
    \left|\sum_{p\leq z}\psi(p)\right|
    \ll_B
    \varphi(d)\frac{z}{(\log z)^{B+3}}.
\]
Finally,
\[
    \varphi(d)\leq d\leq(\log z)^B,
\]
and therefore
\[
    \left|\sum_{p\leq z}\psi(p)\right|
    \ll_B
    \frac{z}{(\log z)^3}.
\]
\end{proof}
\begin{remark}
Note that the implied constant in Lemma \ref{lem:endpoint-small-conductors} is not claimed to be effective.
\end{remark}
\begin{lemma}\label{lem:endpoint-weighted}
Fix $c>4/e^2$ and $M>0$. For sufficiently large $x$, put
\[
L\coloneq\log x,\qquad t\coloneq\log L,\qquad Y\coloneq cL^2,\qquad P\coloneq\pi(Y),
\]
and define 
\begin{equation}\label{eq:endpoint-delta}
    \delta_x
    \coloneq
    \frac{2L}{P\log2}.
\end{equation}
For each positive integer $d$, let $\mathcal V_d(x)$ denote the set of primitive
nonprincipal characters $\psi\bmod d$ satisfying
\begin{equation}\label{eq:endpoint-character-condition}
    \left|\sum_{p\le Y}\psi(p)\right|
    \ge
    (1-\delta_x)P,
\end{equation}
and denote $v_d(x)\coloneq\#\mathcal V_d(x)$. Then,
\begin{equation}\label{eq:endpoint-weighted-bound}
    \sum_{d\le x}\frac{v_d(x)}{d}
    \ll_{c,M}
    \frac{1}{(\log\log x)^M}.
\end{equation}
\end{lemma}
\begin{proof}
Put
\[
    V\coloneq\log Y
      =2t+\log c.
\]
By the prime number theorem,
\begin{equation}\label{eq:endpoint-PNT}
    P
    =
    \frac{Y}{V}(1+o(1)),
    \qquad
    \delta_x
    =
    O_c\left(\frac{t}{L}\right)
    =
    o(1).
\end{equation}
In particular, for sufficiently large $x$,
\[
    0<\delta_x<\frac14.
\]

\medskip
\noindent
Choose
\[
    B\coloneq M+7,
    \qquad
    D_0\coloneq V^B.
\]
Applying Lemma~\ref{lem:endpoint-small-conductors} with $z=Y$, we
obtain, uniformly for every nonprincipal character $\psi\bmod d$ with
$d\leq D_0$,
\[
    \left|\sum_{p\leq Y}\psi(p)\right|
    \ll_B
    \frac{Y}{V^3}.
\]
On the other hand, by \eqref{eq:endpoint-PNT},
\[
    (1-\delta_x)P
    \asymp_c
    \frac{Y}{V}.
\]
Hence, for sufficiently large $x$,
\[
    \frac{Y}{V^3}
    =
    o\left((1-\delta_x)P\right).
\]
It follows that no such character satisfies
\eqref{eq:endpoint-character-condition}.  Thus
\begin{equation}\label{eq:endpoint-small-vanish}
    v_d(x)=0
    \qquad(d\leq D_0).
\end{equation}
We now consider the case $d\ge D_0$. Let $  D_0\leq Q\leq x,$ and put
\[
    R(Q)
    \coloneq
    \sum_{Q\leq d<2Q}v_d(x),
    \qquad
    U\coloneq\log(2Q).
\]
Choose
\begin{equation}\label{eq:endpoint-h}
    h
    \coloneq
    \left\lceil\frac{2U}{V}\right\rceil.
\end{equation}
For sufficiently large $x$ we have $h\geq1$, and
\begin{equation}\label{eq:endpoint-h-bounds}
    Y^h\geq(2Q)^2,\qquad \frac{2U}{V}\leq h <\frac{2U}{V}+1,\qquad hV<2U+V.
\end{equation}
Expanding a power, we write
\begin{equation}\label{eq:endpoint-power}
    \left(\sum_{p\leq Y}\psi(p)\right)^h
    =
    \sum_{m\leq Y^h}a_h(m)\psi(m),
\end{equation}
where $a_h(m)$ denotes the number of ordered $h$-tuples of primes $p_1,\ldots,p_h\leq Y$ such that $   p_1\cdots p_h=m$.
Unique factorization gives
\begin{equation}\label{eq:endpoint-ah}
    \sum_m a_h(m)^2
    \leq
    h!P^h.
\end{equation}
Indeed, after fixing one ordered $h$-tuple, there are at most $h!$
ordered tuples having the same product.
Every character counted by $R(Q)$ satisfies
\[
    \left|
        \sum_{m\leq Y^h}a_h(m)\psi(m)
    \right|^2
    \geq
    \bigl((1-\delta_x)P\bigr)^{2h}.
\]
Since $d/\varphi(d)\geq1$, Lemma~\ref{lem:large-sieve}, applied with
modulus bound $2Q$ and length $\lfloor Y^h\rfloor$, gives
\begin{align*}
    R(Q)\bigl((1-\delta_x)P\bigr)^{2h}
    &\leq
    \sum_{d\leq2Q}\frac{d}{\varphi(d)}
    \sum_{\psi\bmod d}^{*}
    \left|
        \sum_{m\leq Y^h}a_h(m)\psi(m)
    \right|^2
    \leq
    \bigl((2Q)^2+Y^h\bigr)
    \sum_m a_h(m)^2
    \leq
    2Y^h h!P^h,
\end{align*}
where we used \eqref{eq:endpoint-h-bounds} and
\eqref{eq:endpoint-ah}.  Consequently,
\begin{equation}\label{eq:endpoint-moment-bound}
    R(Q)
    \leq
    2h!
    \left(
        \frac{Y}{(1-\delta_x)^2P}
    \right)^h.
\end{equation}
By the prime number theorem and $\delta_x=o(1)$,
\[
    \frac{Y}{(1-\delta_x)^2P}
    \ll_c V.
\]
Since $h!\leq h^h$, \eqref{eq:endpoint-moment-bound} therefore also
gives the estimate
\begin{equation}\label{eq:endpoint-rough}
    R(Q)
    \leq
    (ChV)^h
\end{equation}
for some absolute constant $C>0$ and all sufficiently large $x$.\\
We now estimate $R(Q)$ in three conductor ranges.

\medskip
\noindent
\emph{First range: $D_0\leq Q\leq L^4$.}
Here
\[
    U\leq4t+\log2,
    \qquad
    V=2t+\log c.
\]
Thus, for sufficiently large $x$, $  h\leq5$.
By \eqref{eq:endpoint-rough}, we have
\begin{equation}\label{eq:endpoint-range1}
    R(Q)\ll_c t^5.
\end{equation}

\medskip
\noindent
\emph{Second range: $L^4<Q\leq\exp(\sqrt L)$.}
In this range
\[
    U\geq4t,
    \qquad
    U\leq\sqrt L+\log2.
\]
By \eqref{eq:endpoint-h-bounds}, one has
\[
    \log(ChV)
    \leq
    \log\bigl(C(2U+V)\bigr)
    \leq
    \frac12t+O_c(1).
\]
Using \eqref{eq:endpoint-rough},
\begin{align*}
    \frac{\log R(Q)}{U}
    &\leq
    \left(\frac2V+\frac1U\right)\left(\frac12t+O_c(1)\right)
    \leq
    \frac58+o_c(1).
\end{align*}
Hence, for sufficiently large $x$,
\begin{equation}\label{eq:endpoint-range2}
    R(Q)\ll_c Q^{3/4},
    \qquad
    \frac{R(Q)}Q\ll_c Q^{-1/4}.
\end{equation}

\medskip
\noindent
\emph{Third range: $\exp(\sqrt L)<Q\leq x$.}
Here the leading constants in \eqref{eq:endpoint-moment-bound} are
important.  Put
\begin{equation}\label{eq:endpoint-epsilon}
    \varepsilon_x
    \coloneq
    \log\left(
        \frac{Y}
        {(1-\delta_x)^2PV}
    \right).
\end{equation}
By \eqref{eq:endpoint-PNT}, we have $    \varepsilon_x=o_c(1)$.
Using the logarithmic form of Stirling's formula,
\[
    \log(h!)
    =
    h\log h-h+O(\log(h+1)),
\]
and recalling the definition of $\varepsilon_x$, we obtain from
\eqref{eq:endpoint-moment-bound}
\begin{equation}\label{eq:endpoint-stirling}
    \log R(Q)
    \leq
    h\bigl(
        \log(hV)-1+\varepsilon_x
    \bigr)
    +
    O(\log(h+1)).
\end{equation}
Put
\[
    h_0\coloneq\frac{2U}{V}.
\]
Since $h_0\leq h<h_0+1$, consider the function
\[
    f(s)\coloneq
    s\bigl(\log(sV)-1+\varepsilon_x\bigr).
\]
By the ordinary mean value theorem,
\[
    f(h)-f(h_0)
    =
    f'(\xi)(h-h_0)
\]
for some $\xi\in[h_0,h]$.  Since
\[
    f'(s)=\log(sV)+\varepsilon_x
\]
and, throughout this interval,
\[
    sV\leq hV<2U+V=O_c(L),
\]
we have $f'(s)=O_c(t)$.  As $h-h_0<1$, it follows that
\[
    f(h)=f(h_0)+O_c(t).
\]
Consequently,
\begin{equation}\label{eq:endpoint-stirling2}
    \log R(Q)
    \leq \frac{2U}{V}
    \bigl(\log(2U)-1+\varepsilon_x\bigr)
    +
    O_c(t).
\end{equation}
Since $Q\leq x$, we have $    U\leq L+\log2,$ and hence
\[
    \log(2U)
    \leq
    t+\log2+O(1/L).
\]
Define
\begin{equation}\label{eq:endpoint-ac}
    a_c
    \coloneq
    1+\frac12\log c-\log2.
\end{equation}
The condition $ c>\frac4{e^2}$ implies that $a_c>0$. Since $    V=2t+\log c,$ we obtain
\begin{align*}
    \frac2V
    \bigl(\log(2U)-1+\varepsilon_x\bigr)
    &\leq
    \frac{t+\log2-1+\varepsilon_x+O(1/L)}{t+\frac12\log c}
    =
    1-
    \frac{a_c-\varepsilon_x+O(1/L)}{t+\frac12\log c}
    \leq
    1-\frac{a_c}{4t}
\end{align*}
for sufficiently large $x$.
It follows from \eqref{eq:endpoint-stirling2} that
\[
    \log R(Q)
    \leq
    U-\frac{a_cU}{4t}+O_c(t).
\]
Since $ U\geq\sqrt L$ in the present range, we have $  t=o(U/t)$. Hence there exists $x_0(c)$ such that, for $x\geq x_0(c)$,
\[
    |O_c(t)|
    \leq
    \frac{a_cU}{8t},
\]
from which the following estimate follows:
\[
    \log R(Q)
    \leq
    U-\frac{a_cU}{8t}.
\]
Writing
\[
    \kappa_c\coloneq\frac{a_c}{8}>0,
\]
we obtain
\begin{equation}\label{eq:endpoint-range3}
    \frac{R(Q)}Q
    \leq
    2
    \exp\left(
        -\kappa_c\frac{U}{t}
    \right)
    \leq
    2
    \exp\left(
        -\kappa_c\frac{\sqrt L}{t}
    \right).
\end{equation}
It remains to sum the preceding estimates over the dyadic conductor intervals.
Put
\[
    Q_j\coloneq2^jD_0\qquad(j\geq0),
\]
and keep only those $j$ for which $Q_j\leq x$.
By \eqref{eq:endpoint-small-vanish},
\[
    \sum_{d\leq x}\frac{v_d(x)}d
    \leq
    \sum_{\substack{j\geq0\\Q_j\leq x}}\frac{R(Q_j)}{Q_j}.
\]
For the first range, \eqref{eq:endpoint-range1} gives
\[
    \sum_{Q_j\leq L^4}
    \frac{R(Q_j)}{Q_j}
    \ll_c
    t^5
    \sum_{j\geq0}\frac1{2^jD_0}
    \ll_c
    \frac{t^5}{D_0} 
    \ll_{c,M}t^{-M-2},
\]
where the last inequality follows from the fact that
\[
    D_0=V^{M+7}\asymp_c t^{M+7},
\]
In the second range, \eqref{eq:endpoint-range2} gives
\[
    \frac{R(Q_j)}{Q_j}\ll_c Q_j^{-1/4}.
\]
If $j_0$ is the first index for which $Q_{j_0}>L^4$, then
$Q_{j_0}\asymp L^4$, and hence
\[
\begin{aligned}
    \sum_{L^4<Q_j\leq\exp(\sqrt L)}
    \frac{R(Q_j)}{Q_j}
    &\ll_c
    \sum_{j\geq j_0}Q_j^{-1/4} 
    \ll_c
    L^{-1}\sum_{m\geq0}2^{-m/4}
    \ll_c
    L^{-1}.
\end{aligned}
\]
Finally, there are $O(L)$ dyadic intervals in total, and therefore
\eqref{eq:endpoint-range3} gives for the third range
\[
    \ll_c
    L
    \exp\left(
        -\kappa_c\frac{\sqrt L}{t}
    \right).
\]
For every fixed $M>0$,
\[
    L^{-1}
    +
    L
    \exp\left(
        -\kappa_c\frac{\sqrt L}{t}
    \right)
    =
    o_{c,M}(t^{-M}).
\]
Combining the three ranges gives
\[
    \sum_{d\leq x}\frac{v_d(x)}d
    \ll_{c,M}
    t^{-M},
\]
which is \eqref{eq:endpoint-weighted-bound}.
\end{proof}
\begin{remark}
Lemma~\ref{lem:endpoint-weighted} is closely related to
Lemma~\ref{lem:large-sieve-exceptional}, but the two estimates are
used in different ways.

Lemma~\ref{lem:large-sieve-exceptional} is used in the proof of
Theorem~\ref{thm:main}.  On a dyadic interval $Q\leq d<2Q$, it gives
an estimate of the form $R(Q)\ll Q^\rho,$
where $\rho$ can be chosen arbitrarily close to $2/A$.  This is
combined with Proposition~\ref{prop:flexible-exceptional}, which gives
a saving of approximately $Q^{-\lambda}$ for a primitive character
of conductor $d\asymp Q$, with $\lambda<3/4$.  Thus the proof of
Theorem~\ref{thm:main} requires a comparison of these two powers on
each dyadic conductor interval.

This estimate is not sufficient for the endpoint almost-all problem.
Indeed, if one follows the large-sieve calculation underlying
Lemma~\ref{lem:large-sieve-exceptional} with $A=2$, the resulting
rough exponent approaches $1$.  Thus one would obtain, at best, a
bound of the shape $R(Q)\ll Q^{1+\varepsilon}$.
Since a primitive character of conductor $d\asymp Q$ can arise only
from integers divisible by $d$, its contribution to the exceptional
set is bounded by $x R(Q)/Q$,
which in this form gives no decay in $Q$.  To reach the endpoint, we
therefore keep the more precise information in the high-moment
large-sieve estimate instead of replacing it by a bound of the form
$Q^\rho$.

For the almost-all problem, the natural quantity is
\[
    \sum_{d\leq x}\frac{v_d(x)}{d},
\]
because a primitive character of conductor $d$ can correspond to at
most $x/d$ integers $n\leq x$.  Lemma~\ref{lem:endpoint-weighted} is
designed to estimate precisely this weighted sum.  Its proof also
uses all primes $p\leq Y$, rather than only the primes in $(Y/2,Y]$,
and keep the precise dependence on the high moment.

Conversely, Lemma~\ref{lem:endpoint-weighted} does not replace
Lemma~\ref{lem:large-sieve-exceptional} in the proof of the mean-value
theorem.  There Proposition~\ref{prop:flexible-exceptional} produces
the weaker conductor weight $d^{-\lambda}$, with $\lambda<3/4$.
Since $ d^{-\lambda}=d^{1-\lambda}d^{-1},$
control of
\[
    \sum_{d\leq x}\frac{v_d(x)}{d}
\]
does not by itself give comparable control of
\[
    \sum_{d\leq x}v_d(x)d^{-\lambda}.
\]
Thus the dyadic estimate of
Lemma~\ref{lem:large-sieve-exceptional} is suited to the mean-value
argument, whereas the more precise weighted estimate of
Lemma~\ref{lem:endpoint-weighted} is needed for the endpoint
almost-all result.
\end{remark}
We now combine the preceding weighted estimate with the character
associated to an integer $n$ for which $G(n)$ is large. We first prove the corresponding estimate with $\log x$ in place of
$\log n$.
Fix $c>4/e^2$ and $M>0$, and put
\[
    L\coloneq\log x,\qquad Y\coloneq cL^2,\qquad P\coloneq\pi(Y).
\]
Suppose that $n\leq x,$ $G(n)>Y$.
For sufficiently large $x$, such an $n$ has nontrivial
$\Gamma_n$.  By Lemma~\ref{lem:prime-order}, there is a
nonprincipal character $\chi_n\bmod n$ of prime order such that $  B(\chi_n)=G(n)$. Let $\psi_n\bmod d_n$ be the primitive character inducing $\chi_n$.
Then $d_n\mid n$, and, by Lemma~\ref{lem:order-inducer}, $\psi_n$ is also nonprincipal.\\ Let $p\leq Y$ be prime and suppose that $p\nmid n$.  Since $p\leq Y<G(n)=B(\chi_n),$ the definition of $B(\chi_n)$ gives $\chi_n(p)=1$. Since $\chi_n=\psi_n\chi_{0,n}$ and $\chi_{0,n}(p)=1$, we obtain $\psi_n(p)=1$. Hence the primes $p\leq Y$ for which $\psi_n(p)$ can differ from $1$
must divide $n$.  If $\omega(n)$ denotes the number of distinct prime
divisors of $n$, then
\[
    2^{\omega(n)}
    \leq
    \prod_{p\mid n}p
    \leq n,
\]
and therefore
\[
    \omega(n)
    \leq
    \frac{\log n}{\log2}
    \leq
    \frac{L}{\log2}.
\]
Since $|\psi_n(p)-1|\leq2$ for every prime $p$, we get
\begin{align*}
    \left|
        \sum_{p\leq Y}\psi_n(p)-P
    \right|
    &\leq
    2\omega(n)\leq
    \frac{2L}{\log2}.
\end{align*}
Thus, with $\delta_x$ as in \eqref{eq:endpoint-delta},
\[
    \left|\sum_{p\leq Y}\psi_n(p)\right|
    \geq
    (1-\delta_x)P.
\]
Consequently, $    \psi_n\in\mathcal V_{d_n}(x)$.\\
For every exceptional $n$, choose one such primitive inducing
character.  Since $d_n\mid n$, we obtain
\begin{align}
    \#\{n\leq x:G(n)>c(\log x)^2\}
    &\leq
    \sum_{d\leq x}
    \sum_{\psi\in\mathcal V_d(x)}
    \left\lfloor\frac{x}{d}\right\rfloor 
    \leq
    x\sum_{d\leq x}\frac{v_d(x)}d.
    \label{eq:endpoint-global-count}
\end{align}
Lemma~\ref{lem:endpoint-weighted} therefore gives
\begin{equation}\label{eq:endpoint-global-tail}
    \#\{n\leq x:G(n)>c(\log x)^2\}
    \ll_{c,M}
    \frac{x}{(\log\log x)^M}.
\end{equation}
It remains to replace $\log x$ by $\log n$.  Choose fixed constants
$c_0$ and $\alpha$ such that
\[
    \frac4{e^2}<c_0<c,
    \qquad
    \sqrt{\frac{c_0}{c}}<\alpha<1.
\]
If
\[
    x^\alpha<n\leq x,
\]
then
\[
    c(\log n)^2
    \geq
    c\alpha^2(\log x)^2
    >
    c_0(\log x)^2.
\]
Therefore, after splitting the exceptional integers according as
$n\leq x^\alpha$ or $x^\alpha<n\leq x$, we have
\begin{align*}
    \#\{3\leq n\leq x:G(n)>c(\log n)^2\} \leq x^\alpha + \#\{n\leq x:G(n)>c_0(\log x)^2\}.
\end{align*}
Applying \eqref{eq:endpoint-global-tail}  with $c_0$ in place of $c$,
we obtain
\[
\#\{3\leq n\leq x:G(n)>c(\log n)^2\}\ll_{c,M} x^\alpha+\frac{x}{(\log\log x)^M}.
\]
Since $\alpha<1$,
\[
    x^\alpha=o\left(\frac{x}{(\log\log x)^M}\right).
\]
Hence
\[
    \#\{3\leq n\leq x:G(n)>c(\log n)^2\}\ll_{c,M}\frac{x}{(\log\log x)^M},
\]
as claimed.
\section{Proof of Corollary~\ref{cor:almost-all-GRH}}
Fix $c>4/e^2$ and put $L=\log x$, $Y=cL^2$. Let $3\leq n\leq x$ satisfy $G(n)>Y$, and let $\psi_n\bmod d_n$ be the primitive nonprincipal character associated to $n$ in the proof of Theorem~\ref{thm:almost-all-new}. In particular, $d_n\mid n$, $\psi_n\in\mathcal V_{d_n}(x)$, and $\psi_n(m)=1$ whenever $m\leq Y$ and $(m,n)=1$.

Under GRH, the smoothed explicit formula (see
\cite[Lemma~2.2]{LLS15}), together with the standard estimate
$\sum_{\rho}|\rho|^{-2}\ll\log d_n$ for the nontrivial zeros of
$L(s,\psi_n)$, gives
\[
    \left|\Re\sum_{m\leq Y}\Lambda(m)\psi_n(m)
    \log\left(\frac{Y}{m}\right)\right|\ll\sqrt{Y}\log(d_nY),
\]
where $\Lambda$ denotes the von Mangoldt function and the implied
constant is absolute. On the other hand,
\begin{align*}
    \Re\sum_{m\leq Y}\Lambda(m)\psi_n(m)\log\left(\frac{Y}{m}\right)
    &=\sum_{m\leq Y}\Lambda(m)\log\left(\frac{Y}{m}\right)
      +O\bigl(\omega(n)(\log Y)^2\bigr)=(1+o_c(1))Y,
\end{align*}
uniformly for $n\leq x$. Indeed, only prime powers whose associated prime divides $n$ can contribute to the error, and for each such prime the total contribution is $O((\log Y)^2)$. The last equality follows from the prime number theorem by partial summation and $\omega(n)\leq L/\log 2$, since $L(\log Y)^2=o(Y)$. Comparing the two estimates, we obtain
\[
    \log(d_nY)\gg\sqrt{Y}\asymp_c L.
\]
Since $\log Y=o(L)$, there exists a fixed $0<\eta_c<1$ such that, for sufficiently large $x$, every such conductor satisfies $d_n\geq x^{\eta_c}$.

Thus we may restrict the sum in \eqref{eq:endpoint-global-count}
to conductors $d\geq x^{\eta_c}$. With the notation of
Lemma~\ref{lem:endpoint-weighted}, the dyadic intervals
$[Q_j,2Q_j)$, where $Q_j=2^j x^{\eta_c}$, $j\geq0$ and $Q_j\leq x$,
all belong to the third conductor range for sufficiently large $x$.
By \eqref{eq:endpoint-range3},
\begin{align*}
    \#\{3\leq n\leq x:G(n)>c(\log x)^2\}
    &\leq x\sum_{x^{\eta_c}\leq d\leq x}\frac{v_d(x)}{d}
    \ll_c x\sum_{j:Q_j\leq x}
       \exp\left(-\kappa_c\frac{\log Q_j}{\log L}\right)\\
    &\ll_c xL\exp\left(-\kappa_c\eta_c\frac{L}{\log L}\right)
    \ll_c x\exp\left(-b_c'\frac{L}{\log L}\right),
\end{align*}
where $b_c'=\kappa_c\eta_c/2>0$.
To replace $\log x$ by $\log n$, choose fixed $c_0$ and $\alpha$ such
that $4/e^2<c_0<c$ and $\sqrt{c_0/c}<\alpha<1$, as in the proof of
Theorem~\ref{thm:almost-all-new}. Then
\begin{align*}
    \#\{3\leq n\leq x:G(n)>c(\log n)^2\}
    &\leq x^\alpha+\#\{3\leq n\leq x:G(n)>c_0(\log x)^2\}
    \ll_c x\exp\left(-b_c\frac{\log x}{\log\log x}\right)
\end{align*}
for a suitable $b_c>0$, proving the first assertion.

Finally, by the classical pointwise bound $G(n)\ll(\log n)^2$
under GRH (see \cite{monties,ebach}), we have
\[
    \sum_{3\leq n\leq x}G(n)
    \leq cx(\log x)^2+O_c\left(
    x(\log x)^2\exp\left(-b_c\frac{\log x}{\log\log x}\right)\right).
\]
Dividing by $x(\log x)^2$, taking the limit superior as $x\to\infty$
and then letting $c$ tend to $4/e^2$ from above proves the second assertion.

\section{Proof of Corollary~\ref{cor:quantitative-almost-all}}
Fix $A>2$ and $0<\eta<4A$. It suffices to consider $x$
sufficiently large in terms of $A$ and $\eta$. Put
\[
   L\coloneq\log x,\qquad Y\coloneq L^A,
   \qquad P\coloneq P(Y)=\#\mathcal P(Y).
\]
For each $n\leq x$ with $G(n)>Y$, choose the prime-order
character $\chi_n\bmod n$ given by Lemma~\ref{lem:prime-order},
and let $\psi_n\bmod d_n$ be its primitive inducing character.
By Lemma~\ref{lem:large-G-prime-sum}, we have $d_n\mid n$,
$\ord(\psi_n)$ is prime, and
\begin{equation}\label{eq:quant-prime-sum}
    \left|\sum_{q\in\mathcal P(Y)}\psi_n(q)\right|
       \geq \frac12P.
\end{equation}
We count the exceptional integers by splitting the conductors
$d_n$ into three ranges.\\
Choose fixed real numbers $u,\rho,v$ such that
\begin{equation}\label{eq:quant-parameters}
  \eta<u<4A,\qquad \frac2A<\rho<1,\qquad
  v>\max\left\{u,A-2,\frac{\eta}{1-\rho}\right\}.
\end{equation}
All implied constants below may depend on these parameters,
which are fixed in terms of $A$ and $\eta$.

\medskip
\noindent
\emph{Small conductors: $d_n\leq L^u$.}
Recall that $b_r=(r+1)/(4r^2)$. As $u<4A$, we can
choose a fixed integer $r\geq2$ such that
\[
  \gamma\coloneq\frac Ar-u b_r
   =\frac{(4A-u)r-u}{4r^2}>0.
\]
Next choose a fixed positive integer $k$ satisfying
\begin{equation}\label{eq:quant-k-choice}
     k\gamma>2u+\eta+1.
\end{equation}
For sufficiently large $x$, the choice $H=Y=L^A$
satisfies $C_0\log x\leq H\leq x$, where $C_0$
is the constant in Proposition~\ref{prop:flexible-exceptional}.
Applying \eqref{eq:flex-bad-count} with its parameter
$\eta$ set equal to $1$, we obtain, for every primitive
character $\psi\bmod d$ of prime order with $d\leq L^u$,
\[
  \#\{n\leq x:d\mid n,\ B(\psi\chi_{0,n})>Y\}
   \ll_{r,k}\frac{x}{d}(\log Y)^k
                Y^{-k/r}d^{kb_r+1}.
\]
There are at most $\varphi(d)\leq d$ primitive characters
of prime order modulo $d$. Hence the contribution from
exceptional integers whose chosen conductor satisfies
$d_n\leq L^u$ is at most
\begin{align*}
  &\sum_{d\leq L^u}
   \sum_{\substack{\psi\bmod d\ \mathrm{primitive}\\
                    \ord(\psi)\ \mathrm{prime}}}
     \#\{n\leq x:d\mid n,\ B(\psi\chi_{0,n})>Y\}
 \ll_{r,k}
    x(\log Y)^kY^{-k/r}
      \sum_{d\leq L^u}d^{kb_r+1}\\
  &\qquad\ll_{A,\eta}
    x(\log Y)^k L^{-Ak/r+u(kb_r+2)}
 =
    x(\log Y)^k L^{2u-k\gamma}
    \ll_{A,\eta}xL^{-\eta},
\end{align*}
where the last estimate follows from
\eqref{eq:quant-k-choice} and $\log Y=A\log L$.

\medskip
\noindent
\emph{Intermediate conductors: $L^u<d_n\leq L^v$.}
For every positive integer $d$, let $v_d$ be the number of
primitive Dirichlet characters $\psi\bmod d$ satisfying
\[
     \left|\sum_{q\in\mathcal P(Y)}\psi(q)\right|
       \geq\frac12P.
\]
Put $h\coloneq\lceil2v/A\rceil$, which is a fixed
positive integer. Then $Y^h\geq L^{2v}$.
Let $a_h(m)$ denote the number of ordered $h$-tuples of
primes in $\mathcal P(Y)$ whose product is $m$. By unique
factorization,
\[
     \sum_m a_h(m)^2\leq h!P^h.
\]
Moreover,
\[
 \left(\sum_{q\in\mathcal P(Y)}\psi(q)\right)^h
      =\sum_{m\leq Y^h}a_h(m)\psi(m).
\]
Applying Lemma~\ref{lem:large-sieve} with modulus bound
$\lfloor L^v\rfloor$ and length $\lfloor Y^h\rfloor$, we have
\begin{align*}
 \left(\frac P2\right)^{2h}\sum_{d\leq L^v}v_d
 &\leq
 \sum_{d\leq L^v}\frac{d}{\varphi(d)}
 \sum_{\psi\bmod d}^{*}
 \left|\sum_{m\leq Y^h}a_h(m)\psi(m)\right|^2
 \leq (L^{2v}+Y^h)h!P^h
 \leq 2Y^h h!P^h.
\end{align*}
Since $P\asymp Y/\log Y$ by the prime number theorem,
\begin{equation}\label{eq:quant-intermediate-chars}
 \sum_{d\leq L^v}v_d
   \ll 4^h h!\left(\frac YP\right)^h
   \ll_{A,v}(\log Y)^h.
\end{equation}
For every exceptional $n$ in this range, the character
$\psi_n$ satisfies \eqref{eq:quant-prime-sum} and
$d_n\mid n$. Thus \eqref{eq:quant-intermediate-chars} gives
\begin{align*}
  \#\{n\leq x:G(n)>Y,\ L^u<d_n\leq L^v\}
   &\leq x\sum_{L^u<d\leq L^v}\frac{v_d}{d}
   \leq xL^{-u}\sum_{d\leq L^v}v_d
   \ll_{A,v}xL^{-u}(\log Y)^h
    \ll_{A,\eta}xL^{-\eta},
\end{align*}
where we used $u>\eta$.

\medskip
\noindent
\emph{Large conductors: $d_n>L^v$.}
The third condition in \eqref{eq:quant-parameters}
ensures that $L^v\geq L^{A-2}$ for sufficiently large $x$.
For $Q_j\coloneq2^jL^v\leq x$, Lemma~
\ref{lem:large-sieve-exceptional}, with $c=1$, gives
\[
       \sum_{Q_j\leq d<2Q_j}v_d
           \ll_{A,\rho}Q_j^\rho.
\]
Again counting multiples of the conductors, we get
\begin{align*}
  \#\{n\leq x:G(n)>Y,\ d_n>L^v\}
   &\leq x\sum_{L^v<d\leq x}\frac{v_d}{d}
   \ll_{A,\rho}x
      \sum_{\substack{j\geq0\\Q_j\leq x}}Q_j^{\rho-1}
   \ll_{A,\rho}xL^{-v(1-\rho)}
    \ll_{A,\eta}xL^{-\eta},
\end{align*}
where we used $\rho<1$ and $v(1-\rho)>\eta$.
Combining the three ranges proves
\[ \#\{n\leq x:G(n)>(\log x)^A\} \ll_{A,\eta}\frac{x}{(\log x)^\eta}.
\]

\section{Proof of Corollary~\ref{cor:prime-moduli-endpoint}}
Fix $A>2$. We first bound the exceptional primes in a dyadic
interval $[T,2T)$, with $T$ sufficiently large in terms of $A$.
Put
\[
U:=\log T,\qquad Y:=U^A,\qquad V:=\log Y=A\log U,
\qquad P:=\pi(Y),
\]
and let
\[
\mathcal E(T):=\bigl\{p\in[T,2T):p\text{ prime},\
G(p)>(\log p)^A\bigr\}.
\]
For each $p\in\mathcal E(T)$, Lemma~\ref{lem:prime-order}
provides a nonprincipal character $\chi_p\pmod p$ of prime order
such that $B(\chi_p)=G(p)$. Since $p$ is prime, $\chi_p$
is primitive. Furthermore, $Y<T\le p$ for sufficiently large $T$.
For every prime $q\le Y$ we have
\[
q\le Y=(\log T)^A\le(\log p)^A<G(p)=B(\chi_p).
\]
It follows from the definition of $B(\chi_p)$ that
$\chi_p(q)\in\{0,1\}$. As $q<p$, we have $\chi_p(q)\ne0$,
and hence $\chi_p(q)=1$. Thus
\begin{equation}\label{eq:prime-moduli-maximal}
\sum_{\substack{q\le Y\\q\text{ prime}}}\chi_p(q)=P
\qquad(p\in\mathcal E(T)).
\end{equation}
Choose
\[
h:=\left\lfloor\frac{U}{(A-1)\log U}\right\rfloor.
\]
For sufficiently large $T$ we have $h\ge1$ and $Y^h\le T^2$.
Define
\[
a_h(m):=\#\bigl\{(q_1,\ldots,q_h):
q_i\le Y\text{ prime},\ q_1\cdots q_h=m\bigr\}.
\]
Unique factorization gives
\begin{equation}\label{eq:prime-moduli-ah}
\sum_{m\le Y^h}a_h(m)^2\le h!P^h.
\end{equation}
Indeed, for a fixed ordered $h$-tuple of primes there are at most
$h!$ ordered tuples with the same product, while there are $P^h$
ordered tuples in total.
By \eqref{eq:prime-moduli-maximal}, each $p\in\mathcal E(T)$
satisfies
\[
\sum_{m\le Y^h}a_h(m)\chi_p(m)=P^h.
\]
These characters are primitive and have distinct conductors.
Put
\[
E:=\#\mathcal E(T),\qquad N:=\lfloor Y^h\rfloor.
\]
We now derive a large-sieve type inequality for the
characters $\chi_p$ associated with the exceptional primes.
If $p\ne p'$ belong to $\mathcal E(T)$, then
$\chi_p\overline{\chi_{p'}}$ is a primitive nonprincipal
character modulo $pp'$. By the P\'olya--Vinogradov
inequality (see, for example, \cite{IwaniecKowalski}),
\begin{equation}\label{eq:prime-moduli-offdiagonal}
\left|\sum_{m\le N}
\chi_p(m)\overline{\chi_{p'}(m)}\right|
\ll \sqrt{pp'}\log(pp')
\ll TU.
\end{equation}
For $p=p'$, the same inner sum is at most $N$.

Let $(b_m)_{m\le N}$ be a sequence of complex numbers,
and put
\[
z_p:=\sum_{m\le N}b_m\chi_p(m),
\qquad
S:=\sum_{p\in\mathcal E(T)}|z_p|^2.
\]
By the Cauchy--Schwarz inequality,
\begin{align*}
S^2
&=\left|
\sum_{m\le N}b_m
\sum_{p\in\mathcal E(T)}
\overline{z_p}\chi_p(m)
\right|^2
\le
\left(\sum_{m\le N}|b_m|^2\right)
\sum_{m\le N}
\left|
\sum_{p\in\mathcal E(T)}
\overline{z_p}\chi_p(m)
\right|^2.
\end{align*}
Expanding the last square gives
\begin{align*}
&\sum_{m\le N}
\left|
\sum_{p\in\mathcal E(T)}
\overline{z_p}\chi_p(m)
\right|^2
=
\sum_{p,p'\in\mathcal E(T)}
\overline{z_p}z_{p'}
\sum_{m\le N}
\chi_p(m)\overline{\chi_{p'}(m)}.
\end{align*}
The diagonal contribution is at most
\[
N\sum_{p\in\mathcal E(T)}|z_p|^2=NS.
\]
For the off-diagonal contribution, using
\eqref{eq:prime-moduli-offdiagonal}
and the Cauchy--Schwarz inequality, we obtain
\begin{align*}
&\sum_{\substack{p,p'\in\mathcal E(T)\\p\ne p'}}
|z_p||z_{p'}|
\left|
\sum_{m\le N}
\chi_p(m)\overline{\chi_{p'}(m)}
\right|
\ll
TU\left(\sum_{p\in\mathcal E(T)}|z_p|\right)^2
\le
CETU\sum_{p\in\mathcal E(T)}|z_p|^2
=CETUS.
\end{align*}
Consequently,
\[
S^2\le
\left(\sum_{m\le N}|b_m|^2\right)
(N+CETU)S.
\]
Dividing by $S$ when $S>0$, the case $S=0$
being trivial, we conclude that
\begin{equation}\label{eq:prime-moduli-dual}
\sum_{p\in\mathcal E(T)}
\left|\sum_{m\le N}b_m\chi_p(m)\right|^2
\le
(N+CETU)\sum_{m\le N}|b_m|^2.
\end{equation}
Applying \eqref{eq:prime-moduli-dual} with
$b_m=a_h(m)$, and using \eqref{eq:prime-moduli-ah}
and \eqref{eq:prime-moduli-maximal}, we obtain
\[
EP^{2h}\le (N+CETU)h!P^h.
\]
Thus, setting
\[
\mathcal M_h:=\frac{P^h}{h!},
\]
we conclude that
\begin{equation}\label{eq:prime-moduli-moment}
E\mathcal M_h\le N+CETU.
\end{equation}
We now estimate $\mathcal M_h$. Put $t:=\log U$.
By the prime number theorem and Stirling's formula,
\[
\log P=At-\log(At)+o(1),
\qquad
\log(h!)=h\log h-h+O(\log(h+1)).
\]
Since
\[
h=\frac{U}{(A-1)t}+O(1),
\qquad
\log h=t-\log t-\log(A-1)+o(1),
\]
we obtain
\begin{align*}
\log\mathcal M_h
&=h\log P-\log(h!)
=h\bigl(\log P-\log h+1\bigr)
+O(\log(h+1))\\&
=h\left(
(A-1)t+1-\log\frac{A}{A-1}+o(1)
\right)+O(t)
=U+
\left(
\frac{1-\log\left(\frac{A}{A-1}\right)}
{A-1}+o_A(1)
\right)\frac{U}{t}.
\end{align*}
Define
\[
\beta_A:=
\frac{1-\log\left(\frac{A}{A-1}\right)}{A-1}.
\]
Since $A>2$, we have $\beta_A>0$. Therefore,
\[
\mathcal M_h
=T\exp\left(
(\beta_A+o_A(1))\frac{U}{t}
\right).
\]
In particular, for sufficiently large $T$,
\[
\mathcal M_h\ge 2CTU.
\]
It follows from \eqref{eq:prime-moduli-moment} that
\[
E\le\frac{2N}{\mathcal M_h}.
\]
On the other hand,
\[
\log N\le h\log Y
\le \frac{A}{A-1}U.
\]
Hence, if $E\ge1$,
\begin{align*}
\log E
&\le \log N-\log\mathcal M_h+O(1)
\le \frac{U}{A-1}
-\left(\beta_A+o_A(1)\right)\frac{U}{t}.
\end{align*}
Thus
\begin{equation}\label{eq:prime-moduli-sharp-dyadic}
\log\#\mathcal E(T)
\le \frac{U}{A-1}
-\left(\beta_A+o_A(1)\right)
\frac{U}{\log U}.
\end{equation}
Taking, for example,
\[
c_A:=\frac{\beta_A}{2}
=\frac{1-\log\left(\frac{A}{A-1}\right)}
{2(A-1)}>0,
\]
we conclude that
\begin{equation}\label{eq:prime-moduli-dyadic-bound}
\#\mathcal E(T)
\ll_A T^{1/(A-1)}
\exp\left(-c_A\frac{\log T}{\log\log T}\right).
\end{equation}
This is also trivially true if $\mathcal E(T)$ is empty.\\
Finally, sum \eqref{eq:prime-moduli-dyadic-bound} over dyadic
intervals $[2^j,2^{j+1})$. For sufficiently large $t$ the function
\[
f(t):=t^{1/(A-1)}
\exp\left(-c_A\frac{\log t}{\log\log t}\right)
\]
satisfies $f(2t)\ge 2^{1/A}f(t)$, so this dyadic sum is
bounded by a constant multiple of its final term. The finitely
many smaller primes are absorbed into the implied constant.
This proves \eqref{eq:prime-moduli-endpoint}.\\
Since $(\log p)^A\le(\log x)^A$ for $p\le x$, the set
of exceptional primes corresponding to $(\log x)^A$
is contained in the one corresponding to $(\log p)^A$,
which proves the final assertion.

\section{Proof of Corollary~\ref{cor:least-primitive-roots}}
We first establish a character-sum criterion for the absence of
small prime primitive roots. The shifted sieve reduces the
number of characters occurring in the primitive-root indicator
to a polynomial loss in $\omega(p-1)$.

\begin{lemma}\label{lem:primitive-root-detection}
Let $p$ be an odd prime, let $2\le Y<p$, and write
\[
P=\pi(Y),\qquad w=\omega(p-1).
\]
If no prime $q\le Y$ is a primitive root modulo $p$, then
there exists a nonprincipal character $\chi\bmod p$ such that
\begin{equation}\label{eq:primitive-root-prime-sum}
\left|\sum_{\substack{q\le Y\\q\ \mathrm{prime}}}\chi(q)\right|
\gg \frac{P}{w^2\log^2(2w)},
\end{equation}
with an absolute implied constant.
\end{lemma}

\begin{proof}
Put $n=p-1$ and $R=\operatorname{rad}(n)$, the product of the
distinct prime divisors of $n$. Fix a primitive root $b$ modulo $p$.
For every prime $q\le Y$, let $e(q)\in\{0,\ldots,p-2\}$ be the
unique integer satisfying $q\equiv b^{e(q)}\pmod p$.
Since
\[
\operatorname{ord}_p(q)=\frac{p-1}{\gcd(e(q),p-1)},
\]
the prime $q$ is a primitive root modulo $p$ precisely when
$\gcd(e(q),p-1)=1$. Define
\[
\nu(q)=\prod_{\substack{\ell\mid R\\\ell\mid e(q)}}\ell.
\]
Thus $\nu(q)\mid R$ and $q$ is a primitive root modulo $p$
if and only if $\nu(q)=1$.
For every $d\mid R$ put
\[
A_d=\#\{q\le Y:q\ \mathrm{prime},\ d\mid\nu(q)\},
\qquad
M=\max_{\chi\ne\chi_{0,p}}
\left|\sum_{\substack{q\le Y\\q\ \mathrm{prime}}}\chi(q)\right|.
\]
Since $d$ is squarefree, $d\mid\nu(q)$ is equivalent to
$d\mid e(q)$. By character orthogonality,
\[
\mathbf{1}_{d\mid e(q)}
=\frac1d\sum_{\substack{\chi\bmod p\\\chi^d=\chi_{0,p}}}\chi(q),
\]
and hence
\begin{equation}\label{eq:primitive-root-Ad}
A_d=\frac Pd+O(M)\qquad(d\mid R),
\end{equation}
with an absolute implied constant, since there are exactly $d$
characters satisfying $\chi^d=\chi_{0,p}$.\\
We apply Martin's shifted-sieve lower bound
\cite[Lemma~7]{Martin1997} to the primes $q\le Y$, each counted
with weight $1$, with sieve parameter $R$, total number of primes $P$,
and $f(d)=d$. By \eqref{eq:primitive-root-Ad}, its hypotheses
are satisfied, and it follows that
\begin{equation}\label{eq:primitive-root-shifted-sieve}
\begin{aligned}
&\#\{q\le Y:q\ \mathrm{prime},\ q\text{ primitive modulo }p\}
\ge
\frac{cP}{\log(2w)}
\prod_{\ell\mid n}\left(1-\frac1\ell\right)-Cw^2M,
\end{aligned}
\end{equation}
for absolute constants $c,C>0$.
Let $\ell_w$ be the $w$-th prime. Since $n$ has $w$
distinct prime factors,
\begin{equation}\label{eq:primitive-root-density}
\begin{aligned}
\prod_{\ell\mid n}\left(1-\frac1\ell\right)
&=\frac{\varphi(n)}n
\ge\prod_{\ell\le\ell_w}\left(1-\frac1\ell\right)
\gg\frac1{\log\ell_w}
\gg\frac1{\log(2w)}.
\end{aligned}
\end{equation}
Here we used the standard estimates for the prime product and for
the $w$-th prime, respectively( see
\cite[Cor.~(3.27), Cor.~(3.13)]{RosserSchoenfeld1962}).
If no prime $q\le Y$ is primitive modulo $p$, the left-hand
side of \eqref{eq:primitive-root-shifted-sieve} is zero. Thus
\eqref{eq:primitive-root-density} gives
$M\gg P/(w^2\log^2(2w))$, proving the lemma.
\end{proof}

We now prove Corollary~\ref{cor:least-primitive-roots}.
Fix an integer $r\ge1$, a real number $A>5r$, and
$\varepsilon>0$. It suffices to consider exceptional primes
in a dyadic interval $[T,2T)$, where $T$ is sufficiently large.
Put
\[
U=\log T,\qquad t=\log U,\qquad
Y=U^A,\qquad P=\pi(Y),
\]
and let
\[
\mathcal E(T)=\{p\in[T,2T):p\ \mathrm{prime},\
             g^*(p)>(\log p)^A\},
\qquad E=\#\mathcal E(T).
\]
For each $p\in\mathcal E(T)$, no prime $q\le Y$ is a
primitive root modulo $p$. The standard bound
\[
\omega(p-1)\ll\frac{\log p}{\log\log p}\ll\frac Ut
\]
and Lemma~\ref{lem:primitive-root-detection} therefore give
a nonprincipal character $\chi_p\bmod p$ for which
\begin{equation}\label{eq:primitive-root-chosen-character}
\left|\sum_{\substack{q\le Y\\q\ \mathrm{prime}}}\chi_p(q)\right|
\ge\frac P D,\qquad D=C_0U^2,
\end{equation}
where $C_0>0$ is an absolute constant. The characters $\chi_p$
are primitive and have distinct conductors.
Let $\kappa>0$ be a fixed parameter, and set
\[
\gamma=\frac{r+1}{2r(A-5r)}+\kappa,\qquad
h=\left\lfloor\gamma\frac Ut\right\rfloor,
\qquad N=\lfloor Y^h\rfloor.
\]
For sufficiently large $T$, one has $h\ge1$.
Define
\[
a_h(m)=\#\{(q_1,\ldots,q_h):q_i\le Y\text{ prime},\
                 q_1\cdots q_h=m\}.
\]
As in \eqref{eq:prime-moduli-ah}, unique factorization gives
\begin{equation}\label{eq:primitive-root-coefficients}
\sum_{m\le N}a_h(m)^2\le h!P^h.
\end{equation}
Also, \eqref{eq:primitive-root-chosen-character} implies
\begin{equation}\label{eq:primitive-root-large-moment}
\left|\sum_{m\le N}a_h(m)\chi_p(m)\right|
=\left|\sum_{\substack{q\le Y\\q\ \mathrm{prime}}}\chi_p(q)\right|^h
\ge(P/D)^h
\quad(p\in\mathcal E(T)).
\end{equation}
We next establish a large-sieve type inequality for these
characters. If $p\ne p'$ belong to $\mathcal E(T)$, then
$\chi_p\overline{\chi_{p'}}$ is a nonprincipal Dirichlet
character modulo the squarefree integer $pp'$. Since
$M(pp')=1$ in the notation of Lemma~\ref{norton}, its factor
$R_k(pp')$ equals $1$, regardless of the order $k$ of the
character. Lemma~\ref{norton} therefore gives, for every
$\eta>0$,
\begin{equation}\label{eq:primitive-root-offdiagonal}
\left|\sum_{m\le N}\chi_p(m)\overline{\chi_{p'}(m)}\right|
\ll_{r,\eta}
N^{1-1/r}T^{(r+1)/(2r^2)+\eta}.
\end{equation}
When $p=p'$, the sum on the left is at most $N$.
Let $(a_m)_{m\le N}$ be any complex sequence, and write
\[
z_p=\sum_{m\le N}a_m\chi_p(m),\qquad
S=\sum_{p\in\mathcal E(T)}|z_p|^2.
\]
By Cauchy--Schwarz,
\begin{align*}
S^2
&=\left|\sum_{m\le N}a_m
    \sum_{p\in\mathcal E(T)}\overline{z_p}\chi_p(m)\right|^2
\le\left(\sum_{m\le N}|a_m|^2\right)
    \sum_{m\le N}\left|
     \sum_{p\in\mathcal E(T)}\overline{z_p}\chi_p(m)
    \right|^2.
\end{align*}
Expanding the last square, the diagonal terms contribute at
most $NS$. By \eqref{eq:primitive-root-offdiagonal},
the absolute value of the off-diagonal contribution is
\begin{align*}
&\ll_{r,\eta}N^{1-1/r}T^{(r+1)/(2r^2)+\eta}
  \left(\sum_{p\in\mathcal E(T)}|z_p|\right)^2
\le C_{r,\eta}E N^{1-1/r}T^{(r+1)/(2r^2)+\eta}S,
\end{align*}
where the last inequality follows again from Cauchy--Schwarz.
Consequently, after dividing by $S$ if $S>0$ (the case
$S=0$ being trivial), we obtain
\begin{equation}\label{eq:primitive-root-dual}
\begin{aligned}
&\sum_{p\in\mathcal E(T)}
\left|\sum_{m\le N}a_m\chi_p(m)\right|^2
\le
\left(N+C_{r,\eta}E N^{1-1/r}
T^{(r+1)/(2r^2)+\eta}\right)
\sum_{m\le N}|a_m|^2.
\end{aligned}
\end{equation}
Apply \eqref{eq:primitive-root-dual} with $a_m=a_h(m)$.
By \eqref{eq:primitive-root-coefficients} and
\eqref{eq:primitive-root-large-moment},
\[
E(P/D)^{2h}
\le\left(N+C_{r,\eta}E N^{1-1/r}
T^{(r+1)/(2r^2)+\eta}\right)h!P^h.
\]
Set
\[
\mathcal M_h=\frac{P^h}{D^{2h}h!}.
\]
Then
\begin{equation}\label{eq:primitive-root-moment-bound}
E\mathcal M_h
\le N+C_{r,\eta}E N^{1-1/r}T^{(r+1)/(2r^2)+\eta}.
\end{equation}
By the prime number theorem and Stirling's formula,
\[
\log P=At-\log(At)+o(1),\qquad
\log(h!)=h\log h-h+O(\log(h+1)).
\]
Since $h=\gamma U/t+O(1)$ and $D=C_0U^2$, it follows that
\begin{equation}\label{eq:primitive-root-M-asymptotic}
\begin{aligned}
\log\mathcal M_h
&=h(\log P-2\log D)-\log(h!)
=\gamma(A-5)U+O_{A,r,\kappa}(U/t).
\end{aligned}
\end{equation}
Moreover,
\[
\log N=A\gamma U+O_{A,r,\kappa}(t).
\]
Hence
\begin{align*}
\log\mathcal M_h-
\log\left(N^{1-1/r}T^{(r+1)/(2r^2)+\eta}\right)
&=\left(\gamma\left(\frac Ar-5\right)
 -\frac{r+1}{2r^2}-\eta\right)U
 +O_{A,r,\kappa}(U/t)\\
&=\left(\kappa\left(\frac Ar-5\right)-\eta\right)U
 +O_{A,r,\kappa}(U/t).
\end{align*}
Choose $\eta>0$ so small that
$\eta<\kappa(A/r-5)/2$. Then, for $T$ sufficiently large,
\[
\mathcal M_h\ge
2C_{r,\eta}N^{1-1/r}T^{(r+1)/(2r^2)+\eta}.
\]
It follows from \eqref{eq:primitive-root-moment-bound} that
\[
E\le\frac{2N}{\mathcal M_h}.
\]
If $E\ge1$, we consequently obtain
\begin{align*}
\log E
&\le\log N-\log\mathcal M_h+O(1)
\le5\gamma U+O_{A,r,\kappa}(U/t)
=\left(\frac{5(r+1)}{2r(A-5r)}
 +5\kappa+o_{A,r,\kappa}(1)\right)\log T.
\end{align*}
The case $E=0$ is immediate. As $\kappa>0$ may be chosen
arbitrarily small, this proves
\begin{equation}\label{eq:primitive-root-dyadic}
\#\mathcal E(T)\ll_{A,r,\varepsilon}
T^{\frac{5(r+1)}{2r(A-5r)}+\varepsilon}.
\end{equation}
Summing over dyadic intervals establishes
\eqref{eq:least-prime-primitive-root-general}.
Taking $r=1$, we obtain the asserted power saving for
$A>10$. Finally, $g(p)\le g^*(p)$ and
$(\log p)^A\le(\log x)^A$ for $p\le x$, so the remaining
claims follow.

\section*{Acknowledgments}
The author would like to thank A. Languasco for suggesting the topic and for his insightful comments and advice throughout the preparation of this article. The author is grateful to I. Shparlinski for his valuable suggestions and insights concerning Corollaries~\ref{cor:quantitative-almost-all}, ~\ref{cor:prime-moduli-endpoint} and ~\ref{cor:least-primitive-roots}. The author also thanks P. Moree for his help in improving the presentation of the paper. The author is grateful to Max Planck Institute for Mathematics in Bonn for its hospitality and financial support. The author used ChatGPT (OpenAI, GPT-6.2) for assistance in editing and proofreading the manuscript.

\end{document}